\documentclass[12pt,reqno]{amsart}

\usepackage{amsmath,amsfonts,amsthm,amssymb}

\newtheorem{theorem}{Theorem}[section]
\newtheorem{coro}[theorem]{Corollary}
\newtheorem{lemma}[theorem]{Lemma}
\newtheorem{prop}[theorem]{Proposition}

\theoremstyle{definition}
\newtheorem{definition}[theorem]{Definition}

\theoremstyle{remark}
\newtheorem{remark}[theorem]{Remark}
\usepackage[a4paper,height=20cm,top=24mm,bottom=24mm,left=26mm,right=26mm]{geometry}

\usepackage[numbers,sort]{natbib}
\usepackage[breaklinks=true]{hyperref}

\usepackage{graphicx}
\usepackage{epstopdf}
\usepackage{booktabs}
\usepackage[table]{xcolor}
                                        
\DeclareMathAlphabet{\mathpzc}{OT1}{pzc}{m}{it}
\usepackage{titlesec}
\titleformat{\section}
{\bfseries\scshape\centering}
{\thesection.}{.5em}{}
\titleformat{\subsection}
{\rmfamily\bfseries}
{\thesubsection}{.5em}{}
\titleformat{\subsubsection}
{\rmfamily\bfseries}
{\thesubsubsection}{.5em}{}
\numberwithin{equation}{section}

\newcommand{\I}{\mathrm{i}}

\newcommand{\Z}{\mathbb{Z}}
\newcommand{\ve}{\varepsilon}

\DeclareMathOperator{\CS}{CS}
\DeclareMathOperator{\Li}{Li_2}
\DeclareMathOperator{\Imaginary}{Im}

\DeclareMathOperator{\sign}{sgn}

\DeclareMathOperator{\Vol}{Vol}

\DeclareMathDelimiter{\Norm}{\mathord}{largesymbols}{"3E}{largesymbols}{"3E}

\begin{document}
\baselineskip 16pt
\parskip 8pt
\sloppy


\title[Cluster Algebra \& Complex Volume]{Cluster Algebra and 
  Complex Volume of Once-Punctured Torus Bundles and Two-Bridge Links}


\author[K. Hikami]{Kazuhiro \textsc{Hikami}}

\address{Faculty of Mathematics,
  Kyushu University,
  Fukuoka 819-0395, Japan.}

\email{KHikami@gmail.com}

 \author[R. Inoue]{Rei \textsc{Inoue}}

 \address{Department of Mathematics and Informatics,
   Faculty of Science,
   Chiba University,
   Chiba 263-8522, Japan.}

 \email{reiiy@math.s.chiba-u.ac.jp}


\vspace{18pt}
\date{December 25, 2012. Revised on January 30, 2014.}

\begin{abstract}
We propose a method to compute  complex volume of 2-bridge link
complements.
Our construction sheds light on  a relationship between  cluster variables with coefficients
and canonical decompositions of link complements.

\end{abstract}





\maketitle

\section{Introduction}

The cluster algebra was introduced by Fomin and Zelevinsky 
in~\cite{FominZelev02a},
and it has been studied extensively since then.
The characteristic operation in the
cluster algebra called ``mutation" is related to
various notions, and 
there exist many applications of cluster algebra to
the representation theory of Lie algebras and quantum groups,
triangulated surface~\cite{FomiShapThur08a,FominThurs12a},
Teichm\"uller theory~\cite{FockGonc06b},
integrable systems, and so on.

In geometry,
the cluster algebraic
techniques are  used to understand hyperbolic structure of fibered
bundles~\cite{NagaTeraYama11a},
where  cluster $y$-variables are identified with moduli of ideal hyperbolic
tetrahedra.
Our purpose in this paper is to study complex volume of 2-bridge link
complements $M$ via cluster variables with coefficients.
The complex volume is a complexification of hyperbolic volume,
\begin{equation*}
  \Vol(M)+ \I \, \CS(M) ,
\end{equation*}
where $\Vol(M)$ is the hyperbolic volume
and $\CS(M)$ is the Chern--Simons invariant of~$M$.
Based on  canonical decompositions of 2-bridge link complements
in~\cite{SakumWeek95a},
we  clarify a relationship between ideal tetrahedra and cluster mutations.
Main observation is that the cluster variable with coefficients is
closely related to Zickert's formulation of complex 
volume~\cite{Zicke09}, and that
the complex volume is  given  from the cluster variable
(Theorem~\ref{thm:complex_volume_2bridge}, also Remark~\ref{rem:Ronly}).
We shall also give a formula of complex volume for once-punctured torus
bundle over the circle
(Theorem~\ref{thm:complex_torus}).

There may be natural extensions of our results.
One of them is a quantization of the cluster algebra, which will be
helpful in studies of Volume
Conjecture~\cite{Kasha96b,MuraMura99a},
a relationship between hyperbolic geometry and quantum
invariants.
Indeed in the case of once-punctured torus bundle,
a classical limit of adjoint action
of mutations and its relationship with~\cite{KHikami06a,DimGukLenZag09a} are studied  
in~\cite{TerasYamaz11b}.
Also a generalization to higher rank~\cite{FockGonc06b} remains for
future works.

This paper is organized as follows.
In Section~\ref{sec:review}, we briefly review the definition of  the
cluster algebra and the  three-dimensional hyperbolic geometry, and
explain their
interrelationship by taking a simple example.
Section~\ref{sec:torus} is devoted to the 
once-punctured torus bundles over
the circle.
We formulate the hyperbolic volume via $y$-variables,
and the complex volume via cluster variables.
Section~\ref{sec:knot} is for 2-bridge links.
First we review a canonical decomposition of the  2-bridge link
complements, and we reformulate
it in terms of the cluster algebra.
The key is to introduce  the cluster coefficient
as an element of the tropical semifield.
We give an explicit formula for the complex volume in terms of the
cluster variables.

\section{Cluster Algebra and 3-Dimensional Hyperbolic Geometry}
\label{sec:review}
\subsection{Cluster Algebra}

We briefly give  a definition of  the cluster algebras
following~\cite{FominZelev02a,FominZelev07a}.
See these papers for details.
We let $(\mathbb{P}, \oplus, \cdot)$ be a semifield
with a multiplication $\cdot$ and an addition $\oplus$.
This means that $(\mathbb{P}, \cdot)$ 
is an abelian multiplicative group 
endowed with a binary operation $\oplus$ which is
commutative, associative, and distributive 
with respect to the group multiplication $\cdot$. 
Let $\mathbb{QP}$ denote the quotient field of the group ring 
$\mathbb{ZP}$ of $(\mathbb{P},\cdot)$. 
(Note that $\mathbb{ZP}$ is an integral domain \cite{FominZelev02a}.)
Fix $N \in \Z_{>0}$.
Let $\mathbb{QP}(\boldsymbol{u})$ be the rational functional field of 
algebraically independent variables $\{u_k\}_{k=1,\ldots,N}$.
\begin{definition}
  A seed is a triple
  $(
  \boldsymbol{x},
  \boldsymbol{\varepsilon},
  \mathbf{B}
  )$, where
  \begin{itemize}
  \item
    a cluster $\boldsymbol{x}=(x_1, \dots, x_N)$ is an $N$-tuple 
    of elements in $\mathbb{QP}(\boldsymbol{u})$
    such that $\{x_k\}_{k=1,\ldots,N}$ is a free generating set of
    $\mathbb{QP}(\boldsymbol{u})$, 
  \item
    a coefficient tuple
    $\boldsymbol{\varepsilon}=(\varepsilon_1, \dots, \varepsilon_N)$
    is an $N$-tuple of elements in $\mathbb{P}$,

  \item 
    an exchange matrix
    $\mathbf{B}=(b_{ij})$ is an $N\times N$ skew symmetric integer
    matrix.

  \end{itemize}
We call $x_i$ a cluster variable, and $\varepsilon_i$ a 
coefficient.
\end{definition}

An important tool in the cluster algebra is a mutation, which relates
cluster seeds.

\begin{definition}
  \label{def:mutation_x}
  Let $(
  \boldsymbol{x},
  \boldsymbol{\varepsilon},
  \mathbf{B}
  )$ be a seed.
  For each $k=1,\ldots,N$, we define the mutation of 
  $(\boldsymbol{x}, \boldsymbol{\varepsilon},\mathbf{B})$
  by $\mu_k$ as
  \begin{equation*}
    \mu_k( \boldsymbol{x}, \boldsymbol{\varepsilon}, \mathbf{B})
    =
    (
    \widetilde{\boldsymbol{x}},
    \widetilde{\boldsymbol{\varepsilon}},
    \widetilde{\mathbf{B}}
    ),
  \end{equation*}
  where
  \begin{itemize}
   \item an $N$-tuple 
    $\widetilde{\boldsymbol{x}}=
    (\widetilde{x}_1,\dots, \widetilde{x}_N)$ 
    of elements in $\mathbb{QP}(\boldsymbol{u})$ is
    \begin{equation}\label{eq:x-mutation}
      \widetilde{x}_i
      =
      \begin{cases}
        x_i,
        & \text{for $i \neq k$,}
        \\[2ex]
        \displaystyle
        \frac{\varepsilon_k}{1 \oplus \varepsilon_k} \cdot
        \frac{1}{x_k}
        \prod_{j: b_{jk} >0} x_j^{~b_{jk}}
        +
        \frac{1}{1 \oplus \varepsilon_k} \cdot
        \frac{1}{x_k}
        \prod_{j: b_{jk}<0} x_j^{~-b_{jk}} ,
        &
        \text{for $i = k$,}
      \end{cases}
    \end{equation}
    
  \item a coefficient tuple 
    $\widetilde{\boldsymbol{\varepsilon}}=(
    \widetilde{\varepsilon}_1, \dots,
    \widetilde{\varepsilon}_N)$
    is
    \begin{equation}
      \label{mutation_epsilon}
      \widetilde{\varepsilon}_i
      =
      \begin{cases}
        \varepsilon_k^{~-1} , &
        \text{for $i=k$},
        \\
        \displaystyle
        \varepsilon_i \,
        \left(
          \frac{\varepsilon_k}{
            1 \oplus \varepsilon_k}
        \right)^{b_{ki}}, &
        \text{for $i \neq k$, $b_{ki} \geq 0$,}
        \\[2ex]
        \displaystyle
        \varepsilon_i \,
        \left( 
          1\oplus \varepsilon_k
        \right)^{-b_{ki}},
        &
        \text{for $i \neq k$, $b_{ki} \leq 0$,}
      \end{cases}
    \end{equation}

  \item 
    a skew symmetric integral matrix
    $\widetilde{\mathbf{B}}=
    (\widetilde{b}_{i j})$ is
    \begin{equation}
      \label{mutation_B}
      \widetilde{b}_{ij}
      =
      \begin{cases}
        -b_{ij},
        & \text{for $i=k$ or $j=k$,}
        \\[2ex]
        \displaystyle
        b_{ij} + \frac{1}{2} \left(
          | b_{ik} | \, b_{kj} +
          b_{ik} \, | b_{kj}| 
        \right),
        & \text{otherwise.}
      \end{cases}
    \end{equation}
    
  \end{itemize}
Note that the resulted triple $(
    \widetilde{\boldsymbol{x}},
    \widetilde{\boldsymbol{\varepsilon}},
    \widetilde{\mathbf{B}}
    )$ is again a seed.    
We remark that $\mu_k$ is involutive,
and that $\mu_j$ and $\mu_k$ are commutative if
and only if $b_{jk}= b_{kj}= 0$. 
\end{definition}

By starting from an initial seed 
$(\boldsymbol{x}, \boldsymbol{\varepsilon}, \mathbf{B})$,
we iterate mutations and collect all obtained seeds.  
The cluster algebra 
$\mathcal{A}(\boldsymbol{x}, \boldsymbol{\varepsilon}, \mathbf{B})$ 
is the $\mathbb{ZP}$-subalgebra of the rational function field 
$\mathbb{Q P}(\boldsymbol{u})$ 
generated by all the cluster variables.
In fact, in this paper we do not need 
the cluster algebra itself, but the seeds and the mutations.
Further, we use the following:

\begin{prop}[\cite{FominZelev07a}]
  \label{thm:y_from_xe}
  For a seed 
  $(\boldsymbol{x}, \boldsymbol{\varepsilon}, \mathbf{B})$,
  let $\boldsymbol{y}$ be an
  $N$-tuple $\boldsymbol{y}=(y_1,\dots, y_N)$ in 
  $\mathbb{Q P}(\boldsymbol{u})$ defined as
  \begin{equation}
    \label{y_from_xe}
    y_j = \varepsilon_j \prod_k x_k^{~b_{kj}} .
  \end{equation}
 Then the mutation of 
 $(\boldsymbol{x}, \boldsymbol{\varepsilon}, \mathbf{B})$
 in Def.~\ref{def:mutation_x}  induces 
 a mutation of a pair $(\boldsymbol{y}, \mathbf{B})$,
  \begin{equation}
    \mu_k(
    \boldsymbol{y},
    \mathbf{B}
    )
    =
    (
    \widetilde{\boldsymbol{y}},
    \widetilde{\mathbf{B}}
    ) ,
  \end{equation}
  where

  \begin{itemize}
  \item
    $\widetilde{\boldsymbol{y}}=
    (\widetilde{y}_1,\dots, \widetilde{y}_N)$ is
    analogous to~\eqref{mutation_epsilon},
    \begin{equation}
      \widetilde{y}_i
      =
      \begin{cases}
        y_k^{~-1}, & \text{for $i=k$,}
        \\
        %
        \displaystyle
        y_i \, \left( \frac{y_k}{1+y_k} \right)^{b_{ki}},
        & \text{for $i \neq k$, $b_{ki} \geq 0$,}
        \\[2ex]
        y_i \, \left( 1+y_k \right)^{-b_{ki}} ,
        & \text{for $i \neq k$, $b_{ki} \leq 0$,}
      \end{cases}
    \end{equation}

  \item 
    $\widetilde{\mathbf{B}}=(\widetilde{b}_{ij})$ is~\eqref{mutation_B}.
  \end{itemize}
\end{prop}


This proposition holds for an arbitrary semifield $(\mathbb{P},\oplus,\cdot)$.
In this paper we call $y_i$
a cluster $y$-variable,
or a $y$-variable.
Hereafter we use   the following
tropical semifield~\cite{FominZelev07a}.

\begin{definition}\label{def:P}
  Set $\mathbb{P} = \{\delta^k ~|~ k \in \mathbb{Z}\}$.
  Let $(\mathbb{P}, \oplus, \cdot)$ be
  a semifield generated by a variable $\delta$
  with a multiplication $\cdot$ and an addition $\oplus$,
  \begin{equation}
    \delta^{k_1} \oplus \delta^{k_2}
    =
    \delta^{\min(k_1,k_2)} .
  \end{equation}
\end{definition}



\begin{definition}\label{def:beta}
  We define a map 
  $\psi: \mathbb{P} \, {\longrightarrow} \, \{-1,1\}$,
  given by
  substituting $\delta=-1$ in elements of $\mathbb{P}$.
\end{definition}

For the later use, we introduce the permutation 
acting on seeds.

\begin{definition}\label{def:s}
For $i,j \in \{1,\ldots,N\}$ and $i \neq j$, 
let $s_{i,j}$ be a permutation
of subscripts $i$ and $j$ in seeds.
For example a permuted cluster $s_{i,j}(\boldsymbol{x})$ is defined by
\begin{align*}\label{eq:s-action}
  s_{i,j}(\cdots, x_i, \cdots, x_j, \cdots)
  =
  (\cdots, x_j, \cdots, x_i, \cdots).
\end{align*}
Actions on 
$\boldsymbol{\varepsilon}$
and $\mathbf{B}$ are defined
in the same manner.
They induce an
action on  $\boldsymbol{y}$, and 
$s_{i,j}(\boldsymbol{y})$ has a same form.
\end{definition}

\subsection{Hyperbolic Geometry}

A fundamental object in
the  three-dimensional hyperbolic geometry is an
ideal hyperbolic  tetrahedron $\triangle$
in Fig.~\ref{fig:tetrahedron}~\cite{WPThurs80Lecture}.
The tetrahedron is parameterized by a modulus $z\in \mathbb{C}$,
and each dihedral angle  is given as in the figure.
We mean $z^\prime$ and $z^{\prime\prime}$
for given modulus $z$ by
\begin{align}
  z^\prime &= 1- \frac{1}{z} ,
  &
  z^{\prime \prime}
  & = \frac{1}{1-z} .
\end{align}
The cross section by the horosphere at each vertex
is similar to the triangle in
$\mathbb{C}$ with vertices $0$, $1$, and $z$ as 
in Fig.~\ref{fig:triangle}.
In Fig.~\ref{fig:tetrahedron}, 
we give an orientation to tetrahedron by
assigning a vertex ordering~\cite{Zicke09}, which is crucial in
computing the complex
volume of tetrahedra modulo $\pi^2$.

\begin{figure}[tbhp]
  \centering
  \includegraphics[]{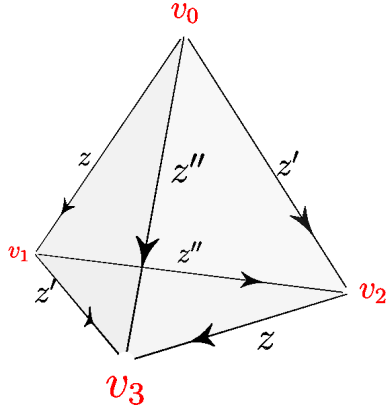}
  \caption{An ideal  hyperbolic tetrahedron $\triangle$ with modulus
    $z$.
    All $4$ vertices are on  $\partial\mathbb{H}^3$,
    and edges are geodesics in $\mathbb{H}^3$.
    Dihedral angles between pairs of faces
    are parametrized by
    $z$, $z^\prime=1-1/z$, and $z^{\prime\prime}=1/(1-z)$.
    To each vertex, we give a  vertex ordering
    $v_0$, $v_1$, $v_2$, and  $v_3$.
    Then an     orientation of $\triangle$  is induced when
    we give an orientation to an edge from $v_a$ to $v_b$ ($a<b$).
  }
  \label{fig:tetrahedron}
\end{figure}

\begin{figure}[tbhp]
  \centering
  \includegraphics[]{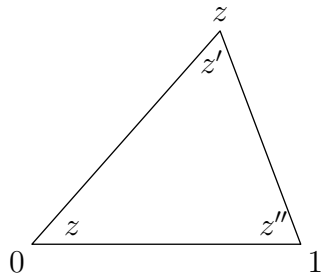}
  \caption{A triangle in $\mathbb{C}$ with vertices $0$, $1$, and $z$.}
  \label{fig:triangle}
\end{figure}

The hyperbolic volume of an ideal tetrahedron $\triangle$ with modulus $z$
is given by the Bloch--Wigner function
\begin{equation}
  \label{BW}
  D(z)
  =
  \Imaginary \Li(z) + \arg(1-z) \, \log|z| .
\end{equation}
Here $\Li(z)$ is the dilogarithm function,
\begin{equation*}
  \Li(z)
  =
  \sum_{n=1}^\infty \frac{z^n}{n^2} ,
\end{equation*}
for $|z|<1$, and the analytic continuation for $z \in \mathbb{C}
\setminus [1,\infty)$
is given by
\begin{equation*}
  \Li(z)
  =
  - \int_0^z \log (1-s) \, \frac{\mathrm{d}s}{s} .
\end{equation*}
Note that
\begin{equation}
  \label{relation_BW}
  \begin{aligned}[b]
    & D(z) = D(z^\prime) = D(z^{\prime\prime})
    \\
    & = - D(1/z) = - D(1/z^\prime) = - D(1/z^{\prime \prime}) .
  \end{aligned}
\end{equation}
See, \emph{e.g.},~\cite{Zagier07a} for details of the dilogarithm function.

We study the case that a cusped hyperbolic $3$-manifold $M$ is
triangulated into
a set of ideal
tetrahedra~$\{\triangle_\nu\}$.
It is known that
the modulus $z_\nu$ of each ideal tetrahedron~$\triangle_\nu$
is determined from two conditions.
One is
a set of  gluing equations, which means
that dihedral angles around each
edge sum up to $2\pi$.
Another is
a cusp  condition so that $M$ has a complete hyperbolic structure.
Then the hyperbolic volume of $M$ is given by
\begin{equation}
  \label{volM}
  \Vol(M) = \sum_\nu D(z_\nu) .
\end{equation}
See~\cite{WPThurs80Lecture,NeumZagi85a,WDNeum92a} for details.

The complex volume, $\Vol(M)+ \I \, \CS(M)$,
of $M$ is   a complexification of~\eqref{volM}, and
in view of the Bloch--Wigner function~$D(z)$~\eqref{BW},
it is natural to study the dilogarithm function $\Li(z)$.
Although,
in contrast to $D(z)$, the dilogarithm
function $\Li(z)$ is a multi-valued function, and
we need 
a ``flattening'',
\emph{i.e.}, the moduli of ideal tetrahedra with
additional parameters $p$ and $q$.

\begin{definition}[\cite{WDNeum00a}]
  A flattening  of an ideal tetrahedron $\triangle$
  is
  \begin{equation}
    \label{log_dihedral}
    (w_0,w_1,w_2)=
    \left(
      \log z + p \, \pi \, \I ,
      -\log(1-z) + q \, \pi \I,
      \log(1-z) - \log z - (p+q) \, \pi \, \I
    \right) ,
  \end{equation}
  where $z$ is the modulus of $\triangle$ and $p,q \in\mathbb{Z}$.
  We use $(z;p,q)$ to denote the flattening of $\triangle$.
\end{definition}

By use of the flattening~$(z;p,q)$,
 we introduce~\cite{WDNeum92a}
an extended Rogers
dilogarithm function
by
\begin{equation}
  \widehat{L}(z;p,q)
  =
  \Li(z)+
  \frac{1}{2} \log z \, \log(1-z) +
  \frac{\pi \, \I}{2} \,
  \left(
    q \log z + p \, \log(1-z)
  \right) - \frac{\pi^2}{6} ,
\end{equation}
where $p, q \in \mathbb{Z}$.
Here and hereafter,
we mean the 
principal branch in the logarithm.
In~\cite{WDNeum00a}, the extended pre-Bloch group is defined as the
free abelian group on
the  flattenings subject to a lifted five-term relation,
and
it is shown that the complex volume can be given as follows.

\begin{prop}[\cite{WDNeum00a}]
  The complex volume of $M$ is 
  \begin{equation}
    \label{complex_volume_M}
    \I \,\left( \Vol(M)+\I \, \CS(M) \right)
    =
    \sum_\nu
    \sign(\nu) \,
    \widehat{L}(z_\nu; p_\nu , q_\nu) 
    \mod \pi^2,
  \end{equation}
  where $(z_\nu; p_\nu , q_\nu)$ is a flattening of $\triangle_\nu$,
  and 
  $\sign(\nu)$ is $1$ (resp. $-1$) when 
  the orientation of $\triangle_\nu$ is 
  same with (resp. opposite to)
  that of~$\triangle$ in Fig.~\ref{fig:tetrahedron}.
\end{prop}

Zickert  clarified that the flattening $(z; p, q)$ can be given by 
complex  parameters assigned to  
the edges of an ideal tetrahedron in the following way.


\begin{prop}[\cite{Zicke09}]
  \label{prop:Zickert}
  For the ideal tetrahedron $\triangle$ in Fig.~\ref{fig:tetrahedron},
  let $c_{ab}$ be a complex parameter on the edge connecting vertices 
  $v_a$ and $v_b$ for $0 \leq a < b \leq 3$. 
  When these complex parameters satisfy 
  \begin{align}
    \label{Zickert_c}
    &
    \frac{c_{03} \, c_{12}}{c_{02} \, c_{13}}
    = \pm z,
    &
    &
    \frac{c_{01} \, c_{23}}{c_{03} \, c_{12}}
    =
    \pm \left( 1 - \frac{1}{z} \right),
    &
    &
    \frac{c_{02} \, c_{13}}{c_{01} \, c_{23}}
    =
    \pm \frac{1}{1-z},
  \end{align}
  the flattening $(z;p,q)$ is given by
  \begin{equation}
    \label{Zickert_log_c}
    \begin{aligned}
      \log z + p \, \pi \, \I 
      & =
      \log c_{03} + \log c_{12} - \log c_{02} - \log c_{13} ,
      \\[1ex]
      -\log(1-z)+ q\, \pi \, \I
      & =
      \log c_{02} + \log c_{13} - \log c_{01} - \log c_{23} .
    \end{aligned}
  \end{equation}
  In gluing tetrahedra to construct a three-manifold $M$,
  edges identified in $M$ are required to
  have the same complex numbers.
\end{prop}

\begin{remark}
  \label{rem:cv}
  It was demonstrated in~\cite{Zicke09}
  that
  these  complex parameters $c_{ab}$ can be read from a developing
  map such as Figs.~\ref{fig:cusp_torus} and~\ref{fig:cusp_sphere_1}.
  To summarize,
  when we have a triangulation
  $\{\triangle_\nu\}$ with modulus $z_\nu$
  of
  a hyperbolic 3-manifold $M$,
  the flattering $(z_\nu; p_\nu, q_\nu)$ can be given from the above
  with edge parameters $c_{ab}$, and
  we get
  the complex volume of $M$ from~\eqref{complex_volume_M}.
\end{remark}

\subsection{Interrelationship}
\label{sec:interrelation}
  

Correspondence between the cluster algebra and the hyperbolic geometry
can be seen in a simple example.\footnote{We thank T. Dimofte.}
We study a triangulated  surface and its flip as in
Fig.~\ref{fig:example_flip}.
A triangulation is related to a 
quiver, \emph{i.e.}, a directed  graph,
where the number of edges in 
the  triangulation is equal to
the fixed number~$N$ in the cluster algebra~\cite{FomiShapThur08a}.
A  flip can be regarded as a 
cluster mutation, as depicted in the figure.
Note that the exchange matrix $\mathbf{B}=(b_{ij})$ can be read from
the quiver by
\begin{equation*}
  b_{ij} =
  \#\{\text{arrows from $i$ to $j$}\}
  -
 \#\{\text{arrows from $j$ to $i$}\} .
\end{equation*}
By definition~\eqref{y_from_xe},  the mutation
$\mu_3(\boldsymbol{y},\mathbf{B})=
(
\widetilde{\boldsymbol{y}},
\widetilde{\mathbf{B}}
)$, is explicitly written as
\begin{equation}
  \label{y_attachment}
  \begin{aligned}
    \widetilde{y}_1 &= y_1 \, (1 + y_3) ,
    \\
    \widetilde{y}_2 &= y_2 \, (1 + y_3^{~-1})^{-1} ,
    \\
    \widetilde{y}_3 &= y_3^{~-1} ,
    \\
    \widetilde{y}_4 &= y_4 \, (1 + y_3^{~-1})^{-1} ,
    \\
    \widetilde{y}_5 &= y_5 \, (1 + y_3) .
  \end{aligned}
\end{equation}

\begin{figure}[tbhp]
  \centering
  \includegraphics[]{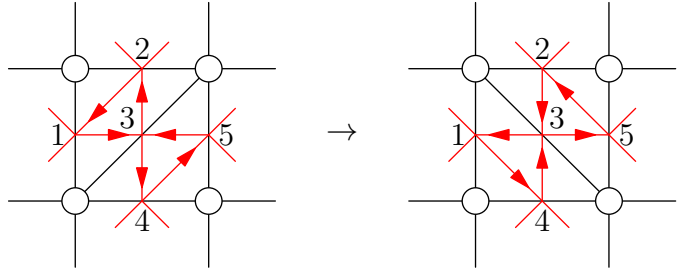}
  \caption{
    A flip of a triangulated  punctured surface.
    Also depicted are the
    quivers associated to the  triangulated surfaces.}
  \label{fig:example_flip}
\end{figure}

\begin{figure}[tbhp]
  \centering
  \includegraphics[scale=0.8]{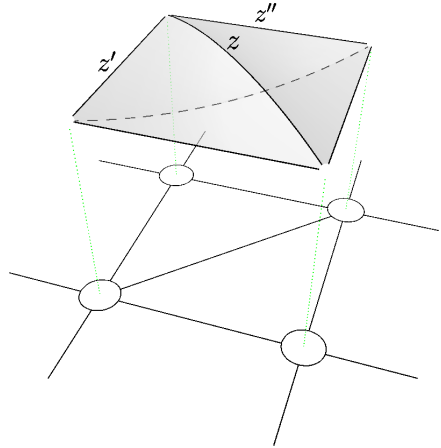}
  \caption{A flip and an attachment of pleated ideal  tetrahedron.}
  \label{fig:attach_tetrahedron}
\end{figure}

On the other hand, we  may regard the  flip  in
Fig.~\ref{fig:example_flip} as 
an attachment of an ideal tetrahedron~$\triangle$ with
modulus~$z$ whose faces are pleated.
See Fig.~\ref{fig:attach_tetrahedron}.
When we denote by $z_k$ the  dihedral angle on edge~$k$ labeled as in
Fig.~\ref{fig:example_flip},  
a dihedral angle $\widetilde{z}_k$  after attaching $\triangle$
is given by
\begin{equation}
  \label{z_attachment}
  \begin{aligned}
    \widetilde{z}_1 & = z_1 \, z^\prime ,\\
    \widetilde{z}_2 & = z_2 \, z^{\prime\prime} ,\\
    \widetilde{z}_3 & = z ,\\
    \widetilde{z}_4 & = z_4 \, z^{\prime \prime} ,\\
    \widetilde{z}_5 & = z_5 \, z^\prime ,\\
  \end{aligned}
\end{equation}
with a hyperbolic gluing condition
\begin{equation*}
  z_3 \, z = 1.
\end{equation*}

Comparing~\eqref{y_attachment} with~\eqref{z_attachment}, we observe
that the cluster $y$-variable is related to
the  dihedral angle by
\begin{equation*}
  y_k = - z_k ,
\end{equation*}
and especially
a  modulus of the ideal
tetrahedron $\triangle$ is given by
\begin{equation}
  \label{z_and_y_mutation}
  z =  - \frac{1}{y_3} .
\end{equation}
See that  a subscript ``3''  denotes 
a label of vertex in the quiver
where
the mutation ($\mu_3$) was applied.


Our idea on a geometrical role of the cluster variable ${x_k}$ 
is as follows.
We know from~\eqref{y_from_xe} that
the $y$-variable is written as
\begin{equation*}
  y_3 = \varepsilon_3 \, \frac{x_1 \, x_5}{x_2 \, x_4}.
\end{equation*}
We also see from \eqref{eq:x-mutation} that 
the mutation $\mu_3$ sends a cluster variable $x_3$ to
$$
  \widetilde{x}_3
  =
  \frac{\ve_3}{1 \oplus \ve_3} \cdot \frac{x_1 \, x_5}{x_3}
  + \frac{1}{1 \oplus \ve_3} \cdot \frac{x_2 \, x_4}{x_3},
$$
where we take the tropical semifield of Def.~\ref{def:P}.
Thus, using \eqref{z_and_y_mutation} we get
\begin{align*}
  &
  z = -\frac{1}{\ve_3} \cdot \frac{x_2 \, x_4}{x_1 \, x_5},
  &
  &
  \frac{1}{1-z} 
  = \frac{\ve_3}{1 \oplus \ve_3} \cdot 
  \frac{x_1 \, x_5}{\widetilde{x}_3 \, x_3}.
\end{align*}
When we apply the map~$\psi$ in  Def.~\ref{def:beta} to the above,
the coefficient parts including~$\varepsilon_3$ are~$\pm 1$.
By comparing these formulae with~\eqref{Zickert_c},
we notice  that the cluster variables $x_i$ play a role of
Zickert's parameters $c_{ab}$.


\section{\mathversion{bold}Once-Punctured Torus Bundle over $S^1$}
\label{sec:torus}
\subsection{\mathversion{bold}$y$-pattern and Hyperbolic Volume}

Let  $\Sigma_{1,1}$  be a once-punctured torus,
$(\mathbb{R}^2 \setminus \mathbb{Z}^2)/\mathbb{Z}^2$.
We set $M_\varphi$
to be the once-punctured torus bundle over the circle, 
whose monodromy is determined by a mapping class 
$\varphi \in SL(2;\mathbb{Z})$.
More precisely, 
via $\varphi: \Sigma_{1,1} \to \Sigma_{1,1}$
we define
an  identification $(x,0) \sim (\varphi(x), 1)$ 
for $\forall x \in \Sigma_{1,1}$, 
and set $M_\varphi = \Sigma_{1,1} \times [0,1]/\sim$.
%
It is known that $M_\varphi$ is hyperbolic when $\varphi$ has distinct
real eigenvalues.
Up to conjugation we have
\begin{equation}
  \label{phi_RL}
  \varphi
  =
  {R}^{s_1} {L}^{t_1} \cdots
  {R}^{s_n} {L}^{t_n}  ,
\end{equation}
where
  \begin{align*}
    & R=
    \begin{pmatrix}
      1 & 1 \\
      0 & 1
    \end{pmatrix},
    &
    & L =
    \begin{pmatrix}
      1 & 0 \\
      1 & 1
    \end{pmatrix} .
  \end{align*}
To denote a mapping class~\eqref{phi_RL}
we use a sequence of symbols
$F_1 \, F_2 \cdots F_{c}
=\underbrace{R \cdots R}_{s_1} \underbrace{L \cdots L}_{t_1} \cdots
\underbrace{R \cdots R}_{s_n} \underbrace{L \cdots L}_{t_n}$
where $F_k=R$ or $L$,
and 
\begin{equation}
  c=\sum_{j=1}^n \left( s_j + t_j \right).
\end{equation}

We use a triangulation of $\Sigma_{1,1}$ as depicted  in
Fig.~\ref{fig:triangulation_1puncture}.
It  is  related to the Farey triangles~\cite{FloyHatc82a}, 
which 
we will explain in Section~\ref{sec:knot_decompose}.
The actions of $R$ and $L$ are interpreted as ``flips'' of
triangulation as shown in the figure.
The triangulation is translated into
the cluster algebra of $N=3$ with the exchange matrix as
\begin{equation}
  \label{B_torus}
  \mathbf{B}=
  \begin{pmatrix}
    0 & -2 & 2 \\
    2 & 0 & -2 \\
    -2 & 2 & 0
  \end{pmatrix} .
\end{equation}
This denotes the quiver in Fig.~\ref{fig:quiver_1puncture}, 
where each vertex has a labeling corresponding to that of 
an edge in the triangulation.
Then the flips $R$ and $L$ can be identified with the mutations
in the cluster algebra (cf. \cite{TerasYamaz11b}),
and we have
\begin{align}
  \label{define_RL}
  {R} & = s_{1,3} \, \mu_1 ,
  &
  {L} & = s_{2,3} \, \mu_2,
\end{align}
where $s_{i,j}$ is the  permutation defined in Def.~\ref{def:s}.
See Fig.~\ref{fig:triangulation_1puncture}. 
We have used  the permutations so that
the exchange matrix~\eqref{B_torus} is invariant under these actions.
In this way the  flips $R$ and $L$ act on
the $y$-variable respectively as
\begin{align}
  &
  (\boldsymbol{y}, \mathbf{B})
  \xrightarrow{R}
  (
  R(\boldsymbol{y}) ,
  \mathbf{B}
  ),
  &
  &
  (\boldsymbol{y}, \mathbf{B})
  \xrightarrow{L}
  (
  L(\boldsymbol{y}) ,
  \mathbf{B}
  ),
\end{align}
where
\begin{align}
  \label{torus_RL_on_y}
  &
  {R} (\boldsymbol{y})
  =
  \begin{pmatrix}
    y_3 ( 1 + y_1^{~-1} )^{-2}
    \\
    y_2 ( 1+ y_1)^2
    \\
    y_1^{~-1}
  \end{pmatrix}^\top ,
  &
  &
  {L}(\boldsymbol{y})
  =
  \begin{pmatrix}
    y_1 ( 1+ y_2^{~-1})^{-2}
    \\
    y_3 (1+y_2)^2
    \\
    y_2^{~-1}
  \end{pmatrix}^\top .
\end{align}

\begin{figure}[tbhp]
  \begin{center}
    \begin{minipage}{0.41\linewidth}
      \centering
      \includegraphics[]{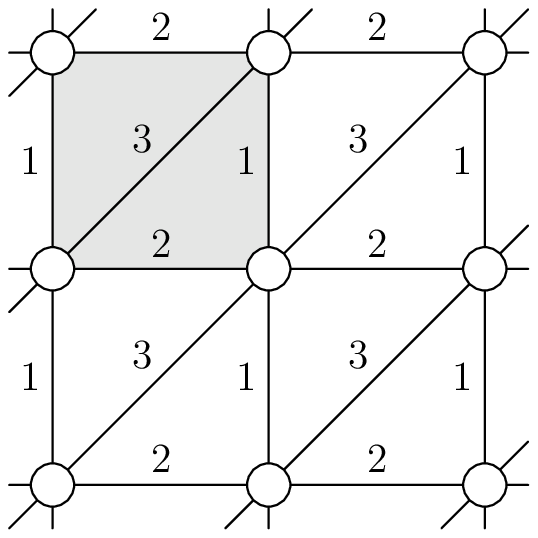}    
    \end{minipage}
    $
    \begin{array}[]{c}
      {R}\\
      \nearrow
      \\[17mm]
      \searrow\\
      {L}
    \end{array}
    $
    \begin{minipage}{0.41\linewidth}
      \centering
      \includegraphics[]{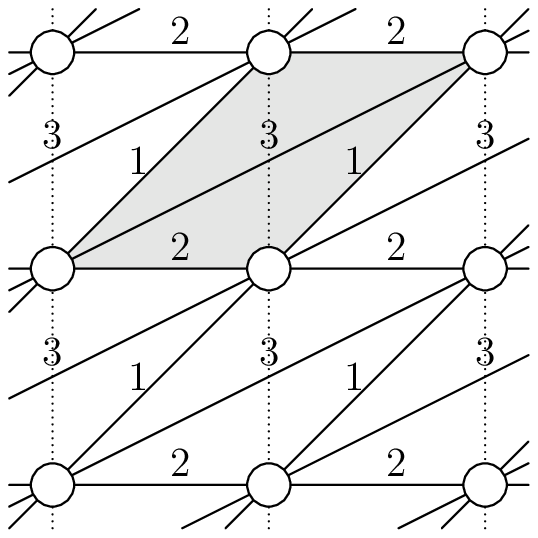}    

      \vspace{11mm}

      \includegraphics[]{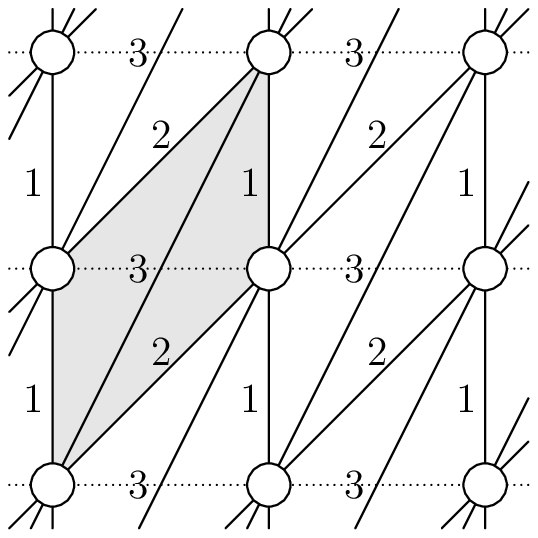}    
    \end{minipage}
  \end{center}
  \caption{A triangulation of the once-punctured torus $\Sigma_{1,1}$
    (left).
    The vertex denotes a puncture.
    A fundamental region is colored gray.
    A labeling of each edge corresponds to that of each vertex 
    in the quiver in Fig.~\ref{fig:quiver_1puncture}.
    The actions of flips, $R$ and $L$, are given in the right hand side.
  }
  \label{fig:triangulation_1puncture}
\end{figure}

\begin{figure}[tbhp]
  \centering
  \includegraphics[]{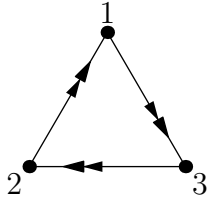}
  \caption{A quiver associated to a triangulation of the once-puncture
    torus $\Sigma_{1,1}$ in the left figure in Fig.~\ref{fig:triangulation_1puncture}.}
  \label{fig:quiver_1puncture}
\end{figure}

\begin{definition}
  A $y$-pattern  of a mapping class 
  $\varphi=F_1 \cdots F_c$~\eqref{phi_RL} is
  $\boldsymbol{y}[k]$ for $k=1,2,\dots, c+1$ defined recursively by
  \begin{equation}
    \label{F_on_y}
    \boldsymbol{y}[k+1]
    =
    F_k\left(
      \boldsymbol{y}[k]
    \right) .
  \end{equation}
\end{definition}

\begin{figure}[tbhp]
  \centering
  \includegraphics[]{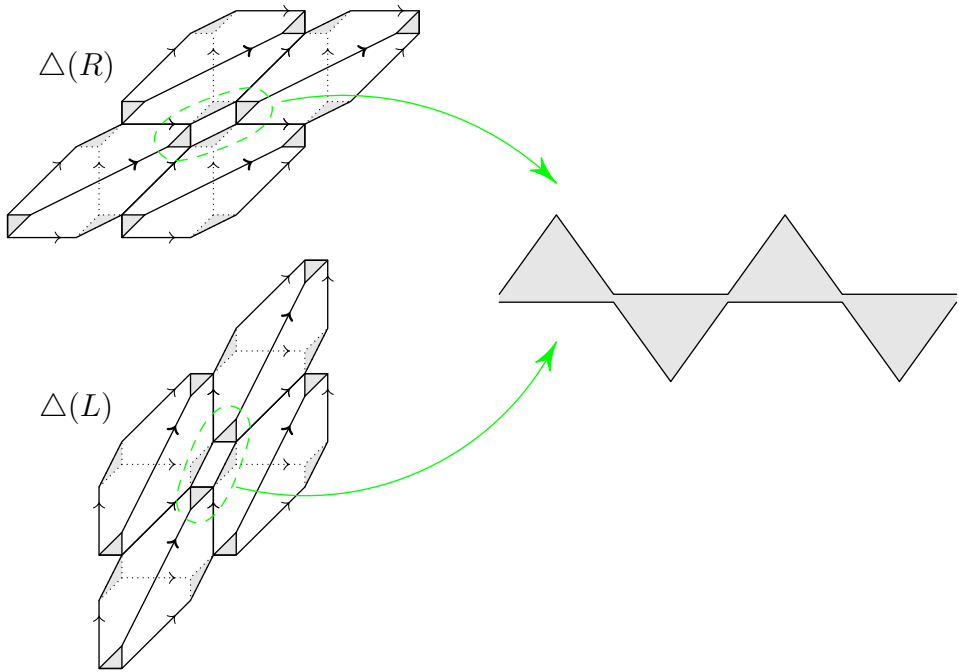}
  \caption{Left: The tetrahedra $\triangle(R)$ and $\triangle(L)$
    assigned to the  flips $R$ and $L$ on the  once-punctured
    torus~$\Sigma_{1,1}$.
    Once we fix an
    orientation of triangulation of $\Sigma_{1,1}$,
    an orientation of tetrahedra is induced as illustrated in the
    figure.
    Right:
    Drawn is a part of Euclidean triangulations of the cusp.}
  \label{fig:torus_4triangle}
\end{figure}

We have seen in Section~\ref{sec:interrelation}
that  the  cluster mutation is interpreted
as an attachment of an ideal tetrahedron.
Thus to each flip $F_k$,
$R$ or $L$ defined in~\eqref{define_RL},
we can assign  a single ideal hyperbolic
tetrahedron
$\triangle(F_k)$
as illustrated  in Fig.~\ref{fig:torus_4triangle}.
The triangulations of $\Sigma_{1,1}$ in
Fig.~\ref{fig:triangulation_1puncture} can be regarded as the top and
bottom pleated faces of ideal tetrahedron in 
Fig.~\ref{fig:torus_4triangle} \cite{FloyHatc82a}.
In a cross section by horosphere at a vertex
we have four triangles, and each of them
has one vertex not shared with any of the other three as in
Fig.~\ref{fig:torus_4triangle}.

Our first claim is that the  modulus of each ideal tetrahedron is given
from a $y$-pattern.
See also~\cite{NagaTeraYama11a}.

\begin{prop}
  \label{thm:volume_torus}
  Let
  $\boldsymbol{y}[k]$ be a
  $y$-pattern of $\varphi$
  with
  an initial condition,
  \begin{equation}
    \label{torus_initial}
    \boldsymbol{y}[1]=
    \left(
      y_1, y_2, \frac{1}{y_1 \, y_2}
    \right) .
  \end{equation}
  Here
  $y_1$ and $y_2$ are  solutions
  of
  \begin{equation}
    \label{torus_periodic}
    \boldsymbol{y}[1]=
    \boldsymbol{y}[c+1] ,
  \end{equation}
  such that
  each modulus $z[k]$ defined by
  \begin{equation}
    \label{torus_z_y}
    z[k] =
    \begin{cases}
      \displaystyle
      -\frac{1}{y[k]_1},
      & \text{when ${F}_k={R}$,}
      \\[3ex]
      \displaystyle
      -\frac{1}{y[k]_2},
      & \text{when ${F}_k = {L}$,}
    \end{cases}
  \end{equation}
  is in the upper half plane $\mathbb{H}$
  for $k=1,\cdots,c$.
  Then $z[k]$ is the modulus of the tetrahedron $\triangle(F_k)$.
\end{prop}

\begin{remark}
  It is known~\cite{Gueri06a}
  that there exists a geometric solution of~\eqref{torus_periodic},
  such that the imaginary part of $z[k]$ is positive for all $k$.
\end{remark}

\begin{coro}
  The hyperbolic volume of $M_\varphi$
  is given by
  \begin{equation}
    \Vol(M_\varphi)
    =
    \sum_{k=1}^c D\left(   z[k]  \right),
  \end{equation}
  where $z[k]$ is the modulus of $\triangle(F_k)$ 
  given in~\eqref{torus_z_y}.
\end{coro}

\subsection{Proof of Proposition~\ref{thm:volume_torus}}
\label{sec:prop_torus}

We have discussed in Section~\ref{sec:review} that the cluster
$y$-variable is related to the
modulus of ideal tetrahedron.
To prove that the $y$-pattern $\boldsymbol{y}[k]$ given
by~\eqref{F_on_y}
describes a complete hyperbolic structure of $M_\varphi$,
we need to check~\cite{WPThurs80Lecture,NeumZagi85a} that
both gluing conditions and a completeness condition are fulfilled.
The gluing conditions may be trivial once we know the fact in
Section~\ref{sec:interrelation}.
Though, we need to check   the gluing conditions at edges on
top/bottom  triangulated surfaces, and the completeness condition is
far from trivial.
In the following,
we shall check all of them by use of
a developing map of torus at infinity of $M_\varphi$~\eqref{phi_RL},
which is depicted  in
Fig.~\ref{fig:cusp_torus}
(see, \emph{e.g.}, \cite{Gueri06a}).

\begin{figure}[tbph]
  \centering
  \includegraphics[scale=0.8]{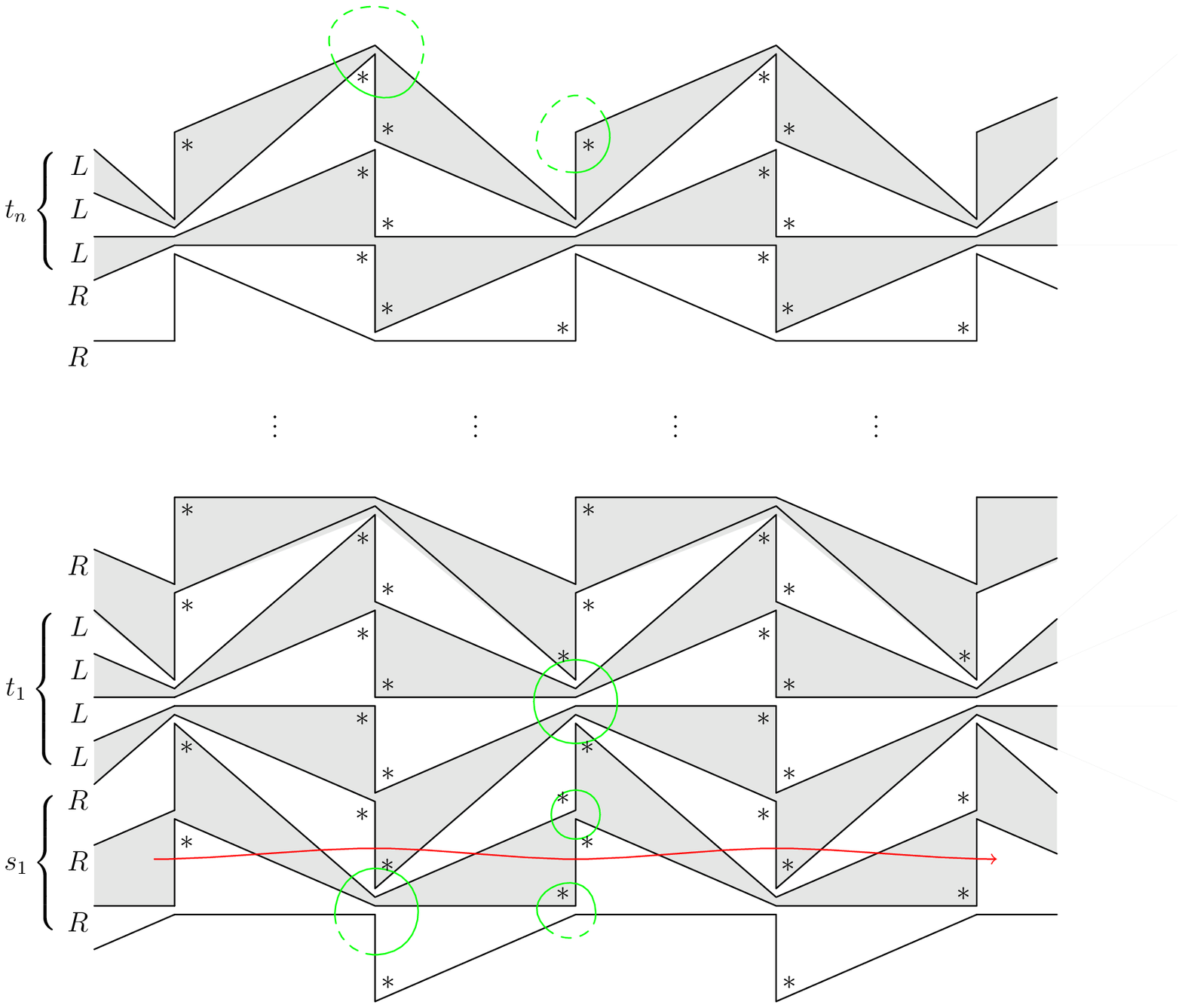}
  \caption{The developing map of the once-punctured torus bundle over the
    circle $M_\varphi$.
    Here $R$ and $L$ respectively denote
    the ideal tetrahedron $\triangle(R)$ and $\triangle(L)$ in
    Fig.~\ref{fig:torus_4triangle}.
    To emphasize the layered structure,
    every vertex is opened up and
    each layer  is colored alternately.
    See Fig.~\ref{fig:torus_4triangle}.
    We denote dihedral angles $z[k]$ by $\ast$. 
  }
  \label{fig:cusp_torus}
\end{figure}

First of all, when $s_1\geq 2$, we have
\begin{align*}
  z[3] \, 
  \left( z^{\prime \prime}[2] \right)^2
  \,
  z[1] 
  &=
  \frac{-1}{y[3]_1} \,
  \left( \frac{1}{1+y[2]_1^{~-1}} \right)^2 \,
  \frac{-1}{y[1]_1}
  \\
  & =
  \frac{-1}{y[2]_3 }
  \,
  \frac{-1}{y[1]_1} = 1.
\end{align*}
Here we have used
$\boldsymbol{y}[3]=R(\boldsymbol{y}[2])$
in the
second equality,
and the last equality follows from
$\boldsymbol{y}[2]=R(\boldsymbol{y}[1])$.
See~\eqref{torus_RL_on_y} for the action of flip~$R$.
We see that this equality is nothing but
a gluing condition for
the second 
circle from the bottom in Fig.~\ref{fig:cusp_torus}.
A case of $s_1=1$ can be checked in the similar manner.

In the same manner, we can check
\begin{equation*}
  z[k+1] \, 
  \left( z^{\prime \prime}[k] \right)^2
  \,
  z[k-1]  =1,
\end{equation*}
for $2\leq k \leq c-1$, which  also corresponds to a gluing
condition in the figure.
Using the periodic
condition~\eqref{torus_periodic}, we have
\begin{equation*}
  \begin{aligned}
    z[2] \,  \left( z^{\prime\prime}[1] \right)^2 \, z[c]
    & = \frac{-1}{y[2]_1} \,
    \left( \frac{1}{1+y[1]_1^{~-1}} \right)^2 \,
    \frac{-1}{y[c]_2}
    \\
    & =1 ,
  \end{aligned}
\end{equation*}
where we have used $y[1]_3=y[c+1]_3=y[c]_2^{~-1}$.
This coincides with a consistency condition for the  right
semi-circle (top and bottom) in the figure.

We can further  check that
\begin{equation*}
  \begin{aligned}
    z[s_1+1]
    \cdot
    \prod_{k=1}^{s_1}  \left( z^\prime[k] \right)^2
    \cdot
    \left( z^\prime[c] \right)^2 \, z[c-1]
    &=
    \frac{-1}{y[s_1+1]_2} \cdot
    \prod_{k=1}^{s_1}
    \left( 1 + y[k]_1 \right)^2 \cdot
    \left( 1+y[c]_2 \right)^2 \,
    \frac{-1}{y[c-1]_2}
    \\
    & =
    \frac{1}{y[1]_2} \, y[c+1]_2
    =1 ,
  \end{aligned}
\end{equation*}
due to~\eqref{torus_periodic}.
This identity corresponds to a gluing condition for  the left 
semi-circle (top and bottom) in the figure.
As another example of this type of equality, we have
\begin{equation*}
  z[s_1+t_1+1] \cdot
  \prod_{k=s_1}^{s_1+t_1} \left( z^{\prime \prime}[k] \right)^2
  \cdot
  z[s_1-1]=1 ,
\end{equation*}
which denotes a gluing condition for the right middle large
circle.
In this way, 
it is straightforward to check  consistency conditions in the figure.

A completeness condition can be checked similarly.  
We have
\begin{equation*}
  \begin{aligned}
    \left(
      \frac{1}{z[1]} \, \frac{z^\prime[2]}{z^{\prime \prime}[2]}
      \cdot
      \prod_{k=3}^{s_1} \left( z^\prime[k] \right)^2
      \cdot
      z[s_1+1]
    \right)^2
    & =
    \left(
      \frac{y[1]_1}{y[2]_1} 
      \cdot
      \prod_{k=2}^{s_1} \left( 1 + y[k]_1 \right)^2
      \cdot
      \frac{1}{y[s_1+1]_2}
    \right)^2
    \\
    & =
    \left(
      \frac{1}{
        y[1]_1 \, y[1]_2 \, y[1]_3
      }
    \right)^2 = 1 ,
  \end{aligned}
\end{equation*}
where the last equality follows from the initial
condition~\eqref{torus_initial}.
We can see that
this equality denotes a completeness condition from the
curve in the figure.

This completes the proof.

\subsection{Cluster Pattern and Complex Volume}

In the previous Section~\ref{sec:prop_torus},
we have proved that the
$y$-variables give the hyperbolic volume of the manifold
$M_\varphi$.
We shall discuss that the cluster variables give the
\emph{complex  volume}
of $M_\varphi$.

We  reformulate the preceding result
in Section~\ref{sec:prop_torus}
by use of cluster variables 
with  coefficients.
By definition,
the permuted mutations $R$ and $L$~\eqref{define_RL}
act on a seed  respectively as
\begin{align}
  &
  (\boldsymbol{x}, \boldsymbol{\varepsilon}, \mathbf{B})
  \xrightarrow{R}
  (R(\boldsymbol{x}, \boldsymbol{\varepsilon}), \mathbf{B}) ,
  &
  (\boldsymbol{x}, \boldsymbol{\varepsilon}, \mathbf{B})
  \xrightarrow{L}
  (L(\boldsymbol{x}, \boldsymbol{\varepsilon}), \mathbf{B}) .
\end{align}
where
\begin{align}
  \label{torus_explicit_xe}
  \begin{split}
  & R(\boldsymbol{x}, \boldsymbol{\varepsilon})
  =
  \left(
    \begin{pmatrix}
      x_3 \\ x_2 \\
      \displaystyle
      \frac{\varepsilon_1}{1 \oplus \ve_1}  \, \frac{x_2^{~2}}{x_1}
      +\frac{1}{1 \oplus \ve_1} \,
      \frac{ x_3^{~2} }{x_1}
    \end{pmatrix}^\top ,
    \begin{pmatrix}
      \varepsilon_3 
      \displaystyle
      \left(\frac{\ve_1}{1 \oplus \ve_1}\right)^2 \\[2ex] 
      \varepsilon_2 (1 \oplus \ve_1)^2\\
      \varepsilon_1^{-1}
    \end{pmatrix}^\top 
  \right) ,
  \\
  & L(\boldsymbol{x}, \boldsymbol{\varepsilon})
  =
  \left(
    \begin{pmatrix}
      x_1 \\ x_3 \\
      \displaystyle
      \frac{1}{1 \oplus \ve_2} \,
      \frac{ x_1^{~2} }{x_2}
      +
      \frac{\varepsilon_2}{1 \oplus \ve_2}  \, \frac{x_3^{~2}}{x_2}
    \end{pmatrix}^\top ,
    \begin{pmatrix}
      \varepsilon_1 
      \displaystyle
      \left(\frac{\ve_2}{1 \oplus \ve_2}\right)^2 \\[2ex]  
      \varepsilon_3 (1 \oplus \ve_2)^2\\
      \varepsilon_2^{-1}
    \end{pmatrix}^\top 
  \right).
  \end{split}
\end{align}

\begin{definition}
  A cluster pattern of  $\varphi=F_1 \cdots F_c$~\eqref{phi_RL} is
  $(\boldsymbol{x}[k], \boldsymbol{\varepsilon}[k])$
  for $k=1,2,\dots,c+1$ defined recursively by
  \begin{gather}
    \label{F_on_x_e}
    (  \boldsymbol{x}[k+1],
    \boldsymbol{\varepsilon}[k+1] )
    =
    F_k\left(
      \boldsymbol{x}[k] ,
      \boldsymbol{\varepsilon}[k]
    \right) .
\end{gather}
\end{definition}

We set an initial seed $(\boldsymbol{x}, \boldsymbol{\varepsilon},
\mathbf{B})$ by
\begin{align}
  &\boldsymbol{x}[1]=(x_1,x_2,x_3),
  & 
  &\boldsymbol{\varepsilon}[1]=(1,1,1),  
\end{align}
and \eqref{B_torus}.
Note that for all $k$ we have
\begin{equation}
  \boldsymbol{\varepsilon}[k]=(1,1,1).
\end{equation}
When we  set a periodic condition
\begin{equation}
  \label{torus_x_condition}
  \boldsymbol{x}[c+1]=    \boldsymbol{x}[1] ,
\end{equation}
all cluster variables $x[k]_j$ are determined up to constant.
Thanks to Prop.~\ref{thm:y_from_xe} with~\eqref{B_torus},
the $y$-pattern in Prop.~\ref{thm:volume_torus}
can be identified with
\begin{equation}
  \label{torus_y_x}
  \boldsymbol{y}[k]
  =
  \left(
    \left( \frac{x[k]_2}{x[k]_3} \right)^2,
    \left( \frac{x[k]_3}{x[k]_1} \right)^2,
    \left( \frac{x[k]_1}{x[k]_2} \right)^2
  \right) ,
\end{equation}
and~\eqref{torus_x_condition} supports the periodicity of $y$-variable,
$\boldsymbol{y}[1]=\boldsymbol{y}[c+1]$.
We choose $\boldsymbol{x}[k]$ such that the
$y$-variable~\eqref{torus_y_x} is a geometric solution
of~\eqref{torus_periodic}.
As was shown in Section~\ref{sec:interrelation},
the cluster variable~$\boldsymbol{x}[k]$
can be regarded as  Zickert's edge
parameters~$c_{ab}$ of $\triangle(F_k)$.
See Remark~\ref{rem:cv}.


To get the complex volume modulo $\pi^2$, we need to take into account
of the  orientation of~$\triangle(F_k)$.
When we assign an    orientation to
the  triangulations of $\Sigma_{1,1}$,
it induces an orientations of $\triangle(F_k)$
as illustrated in
Fig.~\ref{fig:torus_4triangle}~\cite{WDNeum00a}.
The  tetrahedron~$\triangle(R)$ has the same orientation with~$\triangle$
in
Fig.~\ref{fig:tetrahedron}, while 
the  tetrahedron $\triangle(L)$ has 
the opposite orientation.
Using a relationship between the
vertex ordering
and dihedral angles in Fig.~\ref{fig:tetrahedron},
we obtain the following.
\begin{lemma}\label{lem:torus-z-x}
    Let $\boldsymbol{x}[k]$ 
    be a cluster pattern satisfying the condition
    \eqref{torus_x_condition}.
    Then the modulus $z[k]$ of $\triangle(F_k)$ is  
    given by 
  \begin{equation}
    \label{torus_z_x}
    z[k]
    =
    \begin{cases}
      \displaystyle
      \frac{x[k]_1 \, x[k+1]_3}{\left( x[k]_3 \right)^2},
      & \text{for $F_k=R$},
      \\[3ex]
      \displaystyle
      \frac{x[k]_2 \, x[k+1]_3}{\left( x[k]_3 \right)^2},
      & \text{for $F_k=L$},
    \end{cases}
  \end{equation}
  for $k=1,\ldots,c$.
\end{lemma}

\begin{proof}
When $F_k=R$, we see
from  the  vertex ordering of $\triangle(R)$ in
Fig.~\ref{fig:torus_4triangle}
and $\triangle$ in
Fig.~\ref{fig:tetrahedron}
that $-1/y[k]_1$ is identified with $z^{\prime\prime}[k]$.
By using~\eqref{torus_explicit_xe} and~\eqref{torus_y_x}, we obtain
\begin{equation*}
z[k] = 1 + y[k]_1 = \frac{(x[k]_2)^2 + (x[k]_3)^2}{(x[k]_3)^2} 
= \frac{x[k]_1 x[k+1]_3}{(x[k]_3)^2}.
\end{equation*}
When $F_k=L$, we find that
$\triangle(L)$ in Fig.~\ref{fig:torus_4triangle}
has an opposite vertex ordering to $\triangle$ in
Fig.~\ref{fig:tetrahedron}.
Thus $(-1/y[k]_2)^{-1}$ corresponds to $z^{\prime\prime}[k]$,
and we get
\begin{equation*}
z[k] = 1 + \frac{1}{y[k]_2} = \frac{(x[k]_1)^2 + (x[k]_3)^2}{(x[k]_3)^2} 
= \frac{x[k]_2 x[k+1]_3}{(x[k]_3)^2}.
\end{equation*}
\end{proof}


\begin{remark}
  Due to orientation of the tetrahedron,
  a  solution is
  geometric if and only if $\Imaginary z[k]>0$
  (resp. $\Imaginary z[k] <0$)
  for $F_k=R$  
  (resp. $F_k=L$).
\end{remark}



We obtain the flattening of $\triangle(F_k)$ as follows.

\begin{lemma}
  \label{lemma:torus_pq}
  We follow the setting at Lemma~\ref{lem:torus-z-x}.
  The flattening $(z[k]; p[k], q[k])$ for the tetrahedron
  $\triangle(F_k)$
  is
  given by
  \begin{equation}
    \label{torus_flattening_pq}
    \begin{aligned}
      \log z[k] + p[k] \, \pi \, \I
      & =
      \begin{cases}
        \log x[k]_1 + \log x[k+1]_3         -2 \log x[k]_3 ,
        & \text{for $F_k=R$,}
        \\[1ex]
        \log x[k]_2 + \log x[k+1]_3 -2 \log x[k]_3  ,
        & \text{for $F_k=L$,}
      \end{cases}
      \\[2ex]
      -\log \left( 1 - z[k] \right)
      + q[k] \, \pi \, \I
      & =
      \begin{cases}
        2 \log x[k]_3 - 2 \log x[k]_2 ,
        & \text{for $F_k=R$,}
        \\[1ex]
        2 \log x[k]_3 - 2 \log x[k]_1,
        & \text{for $F_k=L$,}
      \end{cases}
    \end{aligned}
  \end{equation}
  for $k=1,\ldots,c$.
\end{lemma}
\begin{proof}
  We recall that the moduli is given in~\eqref{torus_z_x}, and that
  the flips have
  the actions in~\eqref{torus_explicit_xe}.
  Then we obtain
  \begin{equation}\label{eq:torus1-z-x}
    \frac{1}{1-z[k]}
    =
    \begin{cases}
      \displaystyle
      - \left( \frac{x[k]_3}{x[k]_2} \right)^2 ,
      & \text{for $F_k=R$,}
      \\[3ex]
      \displaystyle
      - \left( \frac{x[k]_3}{x[k]_1} \right)^2 ,
      & \text{for $F_k=L$.}
    \end{cases}
  \end{equation}
From \eqref{torus_z_x},~\eqref{eq:torus1-z-x},
and Zickert's identity~\eqref{Zickert_c},
the claim follows.
\end{proof}

As a consequence of Lemma~\ref{lem:torus-z-x} and
Lemma~\ref{lemma:torus_pq}, 
it is straightforward to  obtain the following theorem,
which is the main result in this section.
\begin{theorem}
  \label{thm:complex_torus}
  The complex volume of $M_\varphi$ is 
  \begin{equation}
    \I \, \left(
      \Vol(M_\varphi) + \I \, \CS(M_\varphi)
    \right)
    =
    \sum_{k=1}^c
    \sign(F_k) \,
    \widehat{L}\left(z[k]
    ;
    p[k], q[k] \right)  \mod \pi^2,
  \end{equation}
  where $\sign(R)=1$ and $\sign(L)=-1$, and the flattening 
$(z[k];p[k],q[k])$ is given in~\eqref{torus_z_x} and \eqref{torus_flattening_pq}.
\end{theorem}

\begin{remark}
  \label{rem:torus_orient}
  When we discard
  the vertex ordering of
  the ideal hyperbolic tetrahedra $\triangle(F_k)$,
  the resulting complex volume is defined modulo
  $\frac{\pi^2}{6}$~\cite{WDNeum00a}.
  Ignoring
  orientations of $\triangle(F_k)$, we may simply set
  the moduli of tetrahedra
  from~\eqref{torus_z_y} and~\eqref{torus_y_x}
  as
  \begin{gather*}
    z[k] =
    \begin{cases}
      \displaystyle
      - \left( \frac{x[k]_3}{x[k]_2} \right)^2,
      &
      \text{for $F_k=R$,}
      \\[3ex]
      \displaystyle
      - \left( \frac{x[k]_1}{x[k]_3} \right)^2,
      &
      \text{for $F_k=L$,}
    \end{cases}
  \end{gather*}
  for $k=1,\ldots,c$.
  Then we obtain the flattening $(z[k]; p[k], q[k])$
  of $\triangle(F_k)$ 
  from
  \begin{align*}
    \log z[k] +  p[k] \, \pi \, \I
    & =
    \begin{cases}
      2 \log x[k]_3 - 2 \log x[k]_2,
      & \text{for $F_k=R$,}
      \\[1ex]
      2 \log x[k]_1 - 2 \log x[k]_3,
      & \text{for $F_k=L$,}
    \end{cases}
    \\[2ex]
    -\log(1-z[k])+ q[k] \, \pi \, \I
    &=
    \begin{cases}
      \displaystyle
      2 \log x[k]_2 - \log x[k]_1 -\log x[k+1]_3 ,
      & \text{for $F_k=R$,}
      \\[1ex]
      \displaystyle
      2 \log  x[k]_3 - \log x[k]_2 - \log x[k+1]_3 ,
      & \text{for $F_k=L$.}
    \end{cases}
  \end{align*}
  With these flattening,
  we have
  the complex volume   modulo $\frac{\pi^2}{6}$~\cite{WDNeum00a} by
  \begin{equation*}
    \I \, \left( \Vol(M_\varphi)+ \I \, \CS(M_\varphi) \right)
    =
    \sum_{k=1}^c 
    \widehat{L} \left( z[k]; p[k], q[k] \right) \mod \frac{\pi^2}{6}.
  \end{equation*}
\end{remark}

\subsection{\mathversion{bold}Example: $RL^2$}

We take an example $\varphi=RL^2$.
A cluster pattern is
\begin{equation*}
  \boldsymbol{x}[1]
  \xrightarrow{R}
  \boldsymbol{x}[2]
  \xrightarrow{L}
  \boldsymbol{x}[3]
  \xrightarrow{L}
  \boldsymbol{x}[4] .
\end{equation*}
An initial cluster variable
$\boldsymbol{x}[1]=(x_1,x_2,x_3)$
is solved up to constant from
the  periodic condition~\eqref{torus_x_condition},
$\boldsymbol{x}[1]=\boldsymbol{x}[4]$:
\begin{align*}
  & x_1 = x_3,
  &
  \left( \frac{x_2}{x_3} \right)^{4} +
  \left( \frac{x_2}{x_3} \right)^{2}+2 = 0 .
\end{align*}
By setting $x_1=x_3=1$,
geometric solutions are
$
x_2=
\pm(
0.6760\cdots+
\I \cdot 0.9783\cdots 
)
$.
From~\eqref{torus_flattening_pq} 
the solution with plus sign gives
the flattening parameters~$(p[k],q[k])$ for $k=1,2,3$
as~$(0,-1)$, $(0,1)$, $(0,1)$,
while the minus sign solution  gives
$(0,1)$, $(-2,1)$, $(2,-1)$.
Both solutions give
\begin{equation*}
  \Vol(M_{RL^2})+ \I \, \CS(M_{RL^2})
  =
  2.6667\cdots
  -
  \I \cdot 0.4112\cdots .
\end{equation*}

\section{\mathversion{bold}$2$-Bridge Links}
\label{sec:knot}

\subsection{Canonical Decomposition
  of 2-Bridge Link Complements}
\label{sec:knot_decompose}

We briefly  describe the canonical decomposition of a hyperbolic
2-bridge link.
See~\cite{SakumWeek95a}
(also \cite{Gueri06a}) for details.
We follow  a notation of~\cite{SakumWeek95a}.
Let $K_{q/p}$ be a hyperbolic 2-bridge knot or link
(see, \emph{e.g.}, \cite{Murasugi96Book} for the definition and the
properties). 
Here we assume
that $p$ and $q$ are coprime integers such that
$2 \leq q < p/2$.
When $p$ is odd (resp. even), $K_{q/p}$ is a knot 
(resp. a 2-component link).
We use a continued fraction expression of $q/p$,
\begin{equation}
  q/p=
  \cfrac{1}{a_1+\cfrac{1}{a_2+\cfrac{1}{\ddots + \cfrac{1}{a_n}}}}
  =
  [a_1, a_2, \dots, a_n]
  ,
\end{equation}
where $n \geq 1$,
$a_j \in  \mathbb{Z}_{>0}$,
and $a_n \geq 2$.
We set
\begin{equation}
  c=\sum_{i=1}^n a_i .
\end{equation}
Correspondingly  we have the chain of
the  Farey triangles,
$(\sigma[1], \sigma[2],\dots, \sigma[c])$
as in Fig.~\ref{fig:Farey_2bridge}.
The vertices of each Farey triangle are assigned with rational numbers,
$d_1/d_2$, $d_3/d_4$, and $(d_1+d_3)/(d_2+d_4)$
with $d_i \in \mathbb{Z}$.

\begin{figure}[tbhp]
  \centering
  \includegraphics[width=0.97\textwidth]{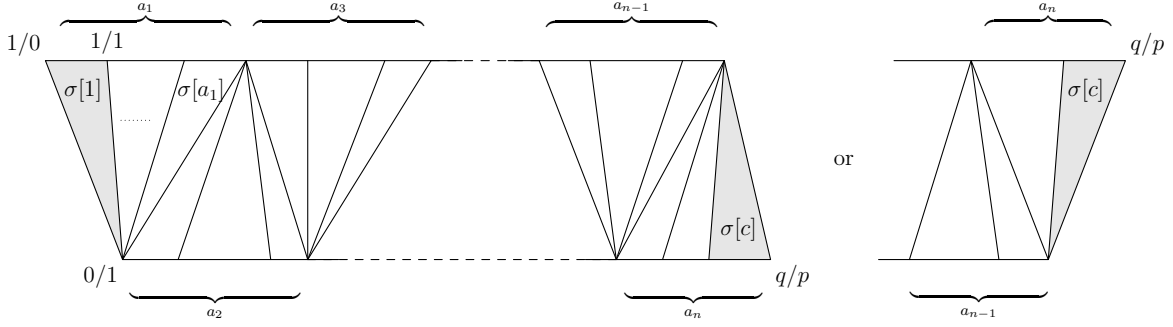}
  \caption{The Farey triangle}
  \label{fig:Farey_2bridge}
\end{figure}

Each Farey triangle determines a triangulation of 4-punctured sphere
$\Sigma_{0,4} =(\mathbb{R}^2\setminus \mathbb{Z}^2)/\Gamma$, where
$\Gamma$ is a transformation group  generated by $\pi$-rotations about
every point in~$\mathbb{Z}^2$.
When two Farey triangles~$\sigma[k]$ and~$\sigma[k+1]$
are adjacent, a flip connects 
triangulations of $\Sigma_{0,4}$ as in Fig.~\ref{fig:triangulation_2bridge}.
Namely when~$\sigma[k+1]$ is in the right (resp. left) to~$\sigma[k]$,
a flip~$R$ (resp.~$L$) acts on the triangulation of~$\Sigma_{0,4}$.
In constructing a canonical  decomposition of 
the link complement
$S^3\setminus K_{q/p}$,
the triangulations for
the first  Farey triangle~$\sigma[1]$  collapses into a single edge,
and  the Farey triangle $\sigma[2]$ is folded along the edge of $1/2$.
Similarly the triangle $\sigma[c]$ collapses into an edge, and
$\sigma[c-1]$ is folded along the edge of $[a_1,\dots,a_n-2]$.
We use a sequence of symbols  to denote these flips as~\cite{SakumWeek95a}
\begin{equation}
  \label{flip_sphere}
  F_1 \, F_2 \cdots F_{c-3}
  =
  \begin{cases}
    R^{a_1-1} \, L^{a_2} \, R^{a_3} \cdots R^{a_{n-1}} \, L^{a_n-2},
    &
    \text{when $n$ is even,}
    \\[2mm]
    R^{a_1-1} \, L^{a_2} \, R^{a_3} \cdots L^{a_{n-1}} \, R^{a_n-2} ,
    &
    \text{when $n$ is odd.}
  \end{cases}
\end{equation}
Here $F_k$ denotes a symbol, $F_k=R$ or $L$,
and acts on the triangulation for $\sigma[k+1]$.

\begin{figure}[tbhp]
  \begin{center}
    \begin{minipage}{0.41\linewidth}
      \centering
      \includegraphics[scale=0.7]{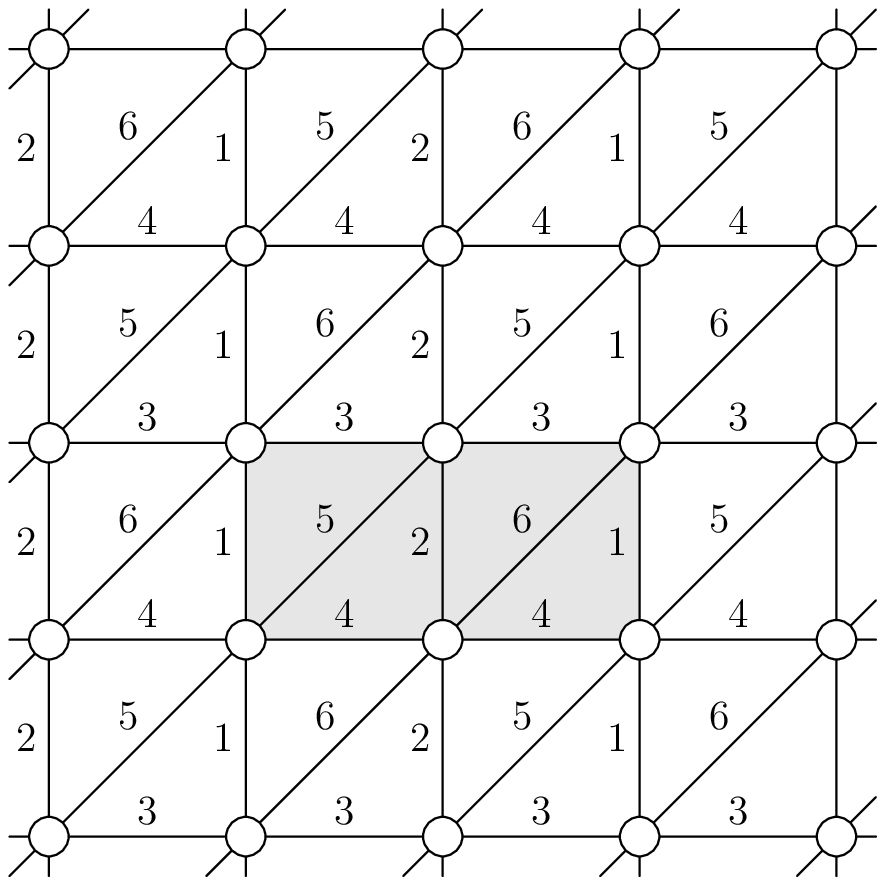}
    \end{minipage}
    $
    \begin{array}[]{c}
      {R}\\
      \nearrow
      \\[17mm]
      \searrow\\
      {L}
    \end{array}
    $
    \begin{minipage}{0.41\linewidth}
      \centering
      \includegraphics[scale=0.7]{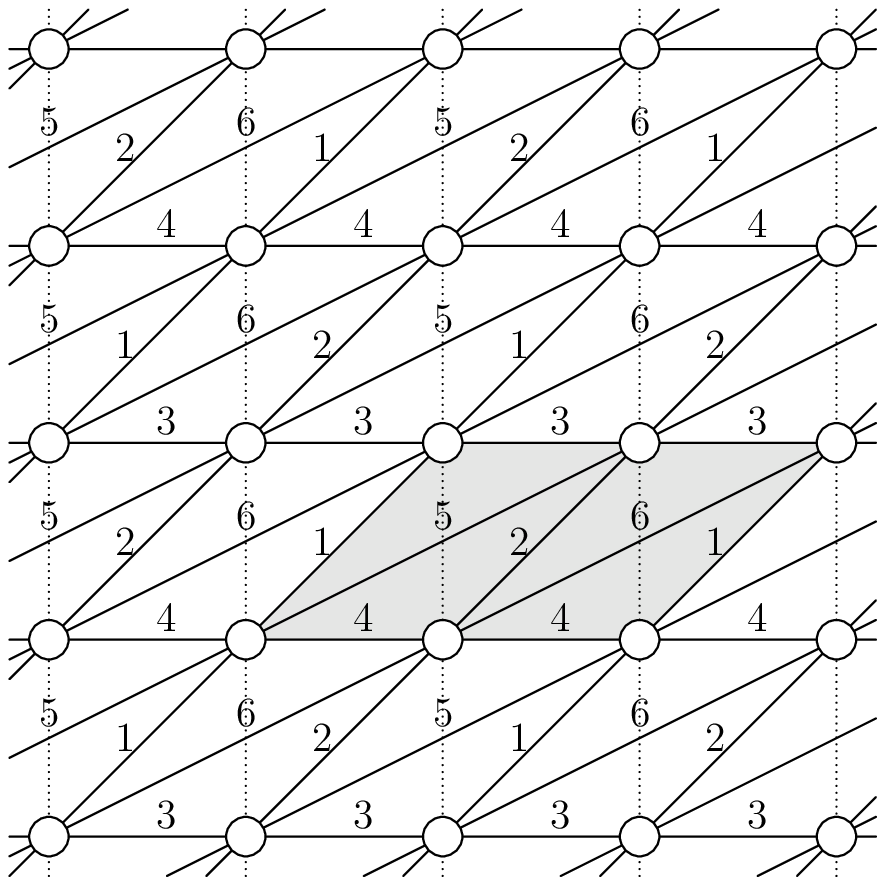}

      \vspace{11mm}

      \includegraphics[scale=0.7]{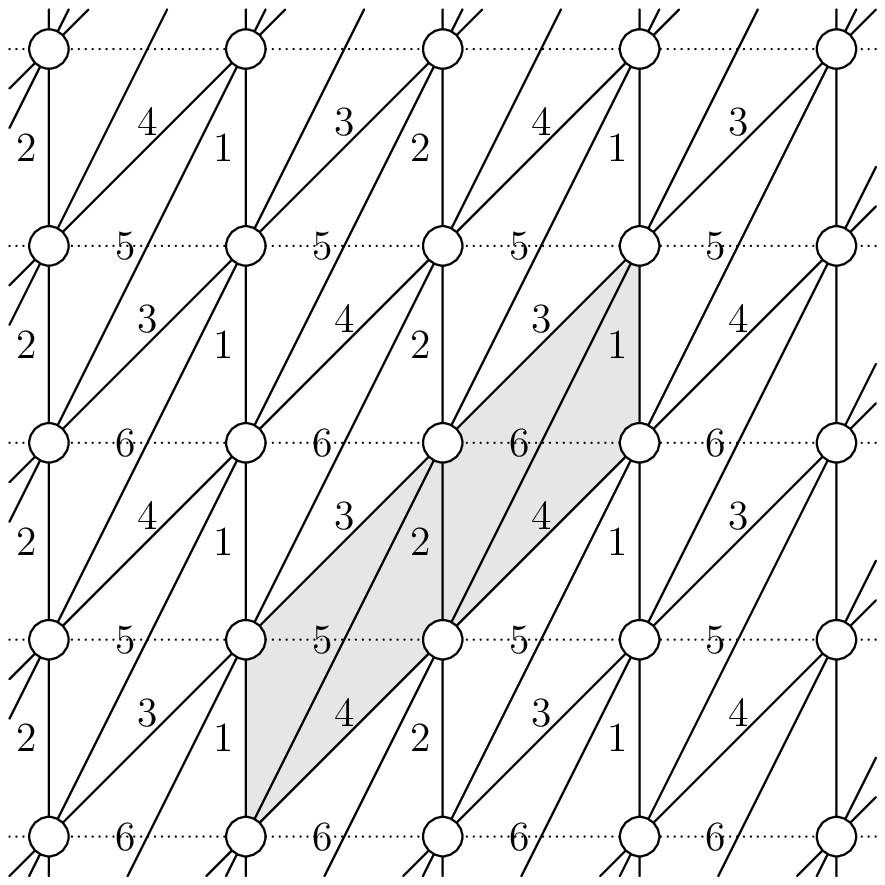}
    \end{minipage}
  \end{center} 
  \caption{A triangulation of 4-punctured sphere $\Sigma_{0,4}$ (left),
    where a
    fundamental region is colored gray.
    The flips of the triangulation are given in the right hand side.
  }
  \label{fig:triangulation_2bridge}
\end{figure}

\begin{figure}[tbhp]
  \centering
  \includegraphics[]{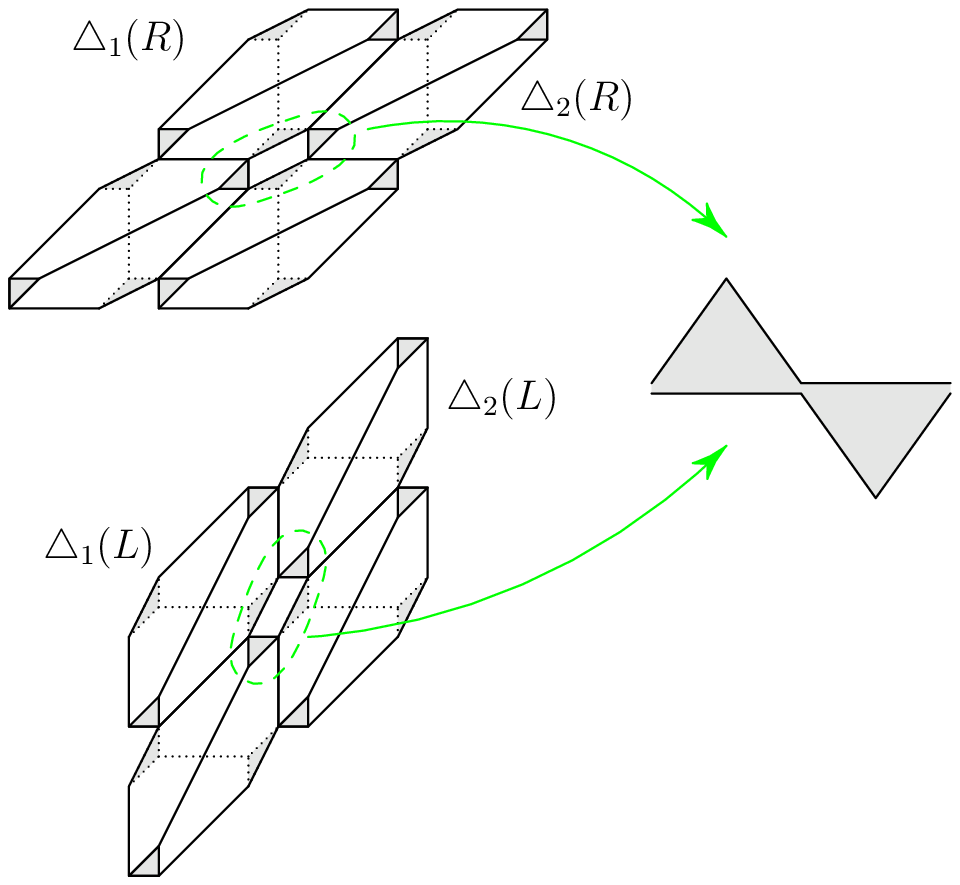}
  \caption{Left:
    The tetrahedra $\triangle_i(R)$ and $\triangle_i(L)$
    assigned to the
    flip $R$ and $L$ on the four-punctured sphere $\Sigma_{0,4}$
    in
    Fig.~\ref{fig:triangulation_2bridge}.
    Right:
    Only a single peripheral annulus is illustrated.}
  \label{fig:sphere_2triangle}
\end{figure}

As in the case of the once-punctured torus bundle, we can assign 
ideal hyperbolic tetrahedra to each flip.
In the case of the four-punctured sphere~$\Sigma_{0,4}$, 
a pair of
ideal hyperbolic tetrahedra,
$\triangle_1(F_k)$ and $\triangle_2(F_k)$,
are associated to the flip $F_k$
as  in
Fig.~\ref{fig:sphere_2triangle}.
Here  we have  illustrated only   a single peripheral annulus, since
the combinatorics of four peripheral annuli are same.
Due to the folding of $\sigma[2]$, the first pair of tetrahedra $\triangle_i(F_1)$
is folded at the edge corresponding to $1/2$.
Similarly the last pair of tetrahedra $\triangle_i(F_{c-3})$ is also
folded at the edge corresponding to $[a_1,\dots,a_{n}-2]$.

\subsection{\mathversion{bold}$y$-pattern and Hyperbolic Volume}

We first formulate the flips~\eqref{flip_sphere} 
by use of the
mutations in the $y$-variables.
We set a triangulation of
the  four-punctured sphere $\Sigma_{0,4}$ as in
Fig.~\ref{fig:triangulation_2bridge}.
It induces the cluster algebra of $N=6$ with 
the exchange matrix $\mathbf{B}$
%
\begin{equation}
  \label{sphere_B}
  \mathbf{B}=
  \begin{pmatrix}
    0 & 0 & -1 & -1 & 1 & 1 \\
    0 & 0 & -1 & -1 & 1 & 1 \\
    1 & 1 & 0 & 0 & -1 & -1 \\
    1 & 1 & 0 & 0 & -1 & -1 \\
    -1 & -1 & 1 & 1 & 0 & 0 \\
    -1 & -1 & 1 & 1 & 0 & 0 
  \end{pmatrix},
\end{equation}
whose  quiver is in
Fig.~\ref{fig:quiver_2bridge}.

\begin{figure}[tbhp]
  \centering
  \includegraphics[]{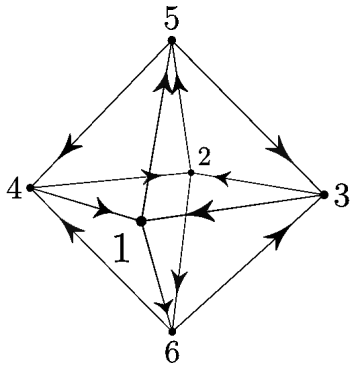}
  \caption{A quiver associated to
    the triangulation of the four-punctured
    sphere $\Sigma_{0,4}$ in Fig.~\ref{fig:triangulation_2bridge}.
    The labeling of  vertex corresponds to that of  edge in
    Fig.~\ref{fig:triangulation_2bridge}.
  }
  \label{fig:quiver_2bridge}
\end{figure}

The flips, $R$ and $L$, in Fig.~\ref{fig:triangulation_2bridge},
are
simply identified with
permuted mutations as
\begin{align}
  \label{sphere_mutation}
    {R}
    & = s_{5,6} \, s_{1,5} \, s_{2,6} \, \mu_1 \, \mu_2 ,
    &
%
    {L}
    & = s_{5,6} \, s_{3,5} \, s_{4,6} \, \mu_3 \, \mu_4 .
  \end{align}
Here,  as in the case of the once-punctured torus, we have inserted
the permutations $s_{i,j}$
so that
the exchange matrix
$\mathbf{B}$~\eqref{sphere_B} is invariant under $R$ and $L$.
Explicit actions on
the $y$-variables are written as
\begin{align}
  &
  (\boldsymbol{y},\mathbf{B})
  \xrightarrow{R}
  (R(\boldsymbol{y}), \mathbf{B}),
  &
  &
  (\boldsymbol{y},\mathbf{B})
  \xrightarrow{L}
  (L(\boldsymbol{y}), \mathbf{B}),
\end{align}
where
\begin{align}
  \label{explicit_y_mutation}
  &
  R(\boldsymbol{y})
  =
  \begin{pmatrix}
    y_5 \, (1 + y_1^{~-1})^{-1} \, ( 1 + y_2^{~-1})^{-1}
    \\
    y_6 \, (1 + y_1^{~-1})^{-1} \, ( 1 + y_2^{~-1})^{-1}
    \\
    y_3 \, (1 + y_1 ) \, (1+y_2) 
    \\
    y_4 \, (1 + y_1 ) \, (1+y_2) 
    \\
    y_2^{~-1}
    \\
    y_1^{~-1}
  \end{pmatrix}^\top
  ,
  &
  &
  L(\boldsymbol{y})
  =
  \begin{pmatrix}
    y_1 \, (1 + y_3^{~-1})^{-1} \, ( 1 + y_4^{~-1})^{-1}
    \\
    y_2 \, (1 + y_3^{~-1})^{-1} \, ( 1 + y_4^{~-1})^{-1}
    \\
    y_5 \, (1 + y_3 ) \, (1+y_4) 
    \\
    y_6 \, (1 + y_3 ) \, (1+y_4) 
    \\
    y_4^{~-1}
    \\
    y_3^{~-1}
  \end{pmatrix}^\top .
\end{align}

\begin{definition}
  A $y$-pattern for $K_{q/p}$ is
  $\boldsymbol{y}[k]$ for $k=1,2, \dots, c-2$ defined by
  \begin{equation}
    \label{pattern_sphere}
    \boldsymbol{y}[k+1]=
    F_k(\boldsymbol{y}[k]) ,
  \end{equation}
  where $F_k$ is $R$ or $L$ as  in~\eqref{flip_sphere}.
\end{definition}

\begin{prop}
  \label{thm:volume_sphere}
  Let $\boldsymbol{y}[k]$ be a $y$-pattern for $K_{q/p}$
  constructed from 
  an  initial condition
  \begin{equation}
    \label{folding_y_bottom}
    \boldsymbol{y}[1]
    =
    \left(
      y,y, -\frac{1}{y}, -\frac{1}{y}, -1, -1
    \right)    .
  \end{equation}
  Here $y$ is a   solution of
  \begin{equation}
    \label{folding_y_top}
    \begin{cases}
      y[c-2]_3 = y[c-2]_4
      =-1,
      & \text{if $n$ is even,}
      \\[2mm]
      y[c-2]_1=y[c-2]_2
      =-1,
      & \text{if $n$ is odd,}
    \end{cases}
  \end{equation}
  such that
  each modulus $z_i[k]$ for $i=1, 2$ and $k=1,2,\dots,c-3$ defined by
  \begin{equation}
    \label{sphere_z_from_y}
    z_i[k]
    =
    \begin{cases}
      \displaystyle
      -\frac{1}{y[k]_i},
      & \text{if $F_k=R$,}
      \\[3ex]
      \displaystyle
      -\frac{1}{y[k]_{2+i}},
      & \text{if $F_k=L$,}
    \end{cases}
  \end{equation}
  is in the upper half plane~$\mathbb{H}$.
  Then $z_i[k]$ denotes the modulus of the tetrahedron
  $\triangle_i(F_k)$.
\end{prop}

\begin{remark}
  \label{rem:cv_bridge}
  See Appendix of~\cite{Gueri06a},
  where  proved is an
  existence of a geometric solution
  of~\eqref{folding_y_top}.
  It is also announced in~\cite{ASWY07}.
\end{remark}

\begin{coro}
  The hyperbolic volume of the link complement 
  $S^3 \setminus K_{q/p}$ is given by
  \begin{equation}
    \Vol(S^3 \setminus K_{q/p})
    =
    \sum_{k=1}^{c-3}
    \sum_{i=1,2}
    D( z_i[k] ),
  \end{equation}
  where
  $D(z)$ is the Bloch--Wigner function~\eqref{BW}, and
  $z_i[k]$'s are the moduli \eqref{sphere_z_from_y}.
\end{coro}


\begin{remark}
  \label{lem:yyy}
  Under the initial condition \eqref{folding_y_bottom},
  we have invariants of $y$-variable under a mutation,
  \begin{align}
    \label{yyy}
    y[k]_{i_1} \, y[k]_{i_2} \, y[k]_{i_3} = 1,
  \end{align}
  for $k=1,2,\ldots,c-2$ and 
  $(i_1, i_2, i_3) \in \{1,2\}\times\{3,4\} \times \{5,6\}$.
\end{remark}
\subsection{Proof of Proposition~\ref{thm:volume_sphere}}

We shall check that the moduli $z_i[k]$ of the ideal hyperbolic
tetrahedra written in terms of
the cluster $y$-variables as~\eqref{sphere_z_from_y}
fulfill both the gluing conditions and the completeness conditions by
use of the developing map.

We study four cases separately,
\begin{enumerate}
  \renewcommand{\theenumi}{\Roman{enumi}}
\item even $n$, and $a_n >2$,
  \label{enum:case1}
  \begin{align*}
    &F_1\cdots F_{c-3}=R^{a_1-1} L^{a_2} \cdots L^{a_n-2},
    &
    &y[c-2]_3=y[c-2]_4=-1,
  \end{align*}

\item even $n$, and $a_n =2$,
  \label{enum:case2}
  \begin{align*}
    &F_1\cdots F_{c-3}=R^{a_1-1} L^{a_2} \cdots R^{a_{n-1}},
    &
    &y[c-2]_3=y[c-2]_4=-1,
  \end{align*}

\item odd $n$, and $a_n >2$,
  \label{enum:case3}
  \begin{align*}
    &F_1\cdots F_{c-3}=R^{a_1-1} L^{a_2} \cdots R^{a_n-2},
    &
    &y[c-2]_1=y[c-2]_2=-1,
  \end{align*}

\item odd $n$, and $a_n =2$,
  \label{enum:case4}
  \begin{align*}
    &F_1\cdots F_{c-3}=R^{a_1-1} L^{a_2} \cdots L^{a_{n-1}},
    &
    &y[c-2]_1=y[c-2]_2=-1.
  \end{align*}
\end{enumerate}

\begin{figure}[tbhp]
  \centering
  \includegraphics[]{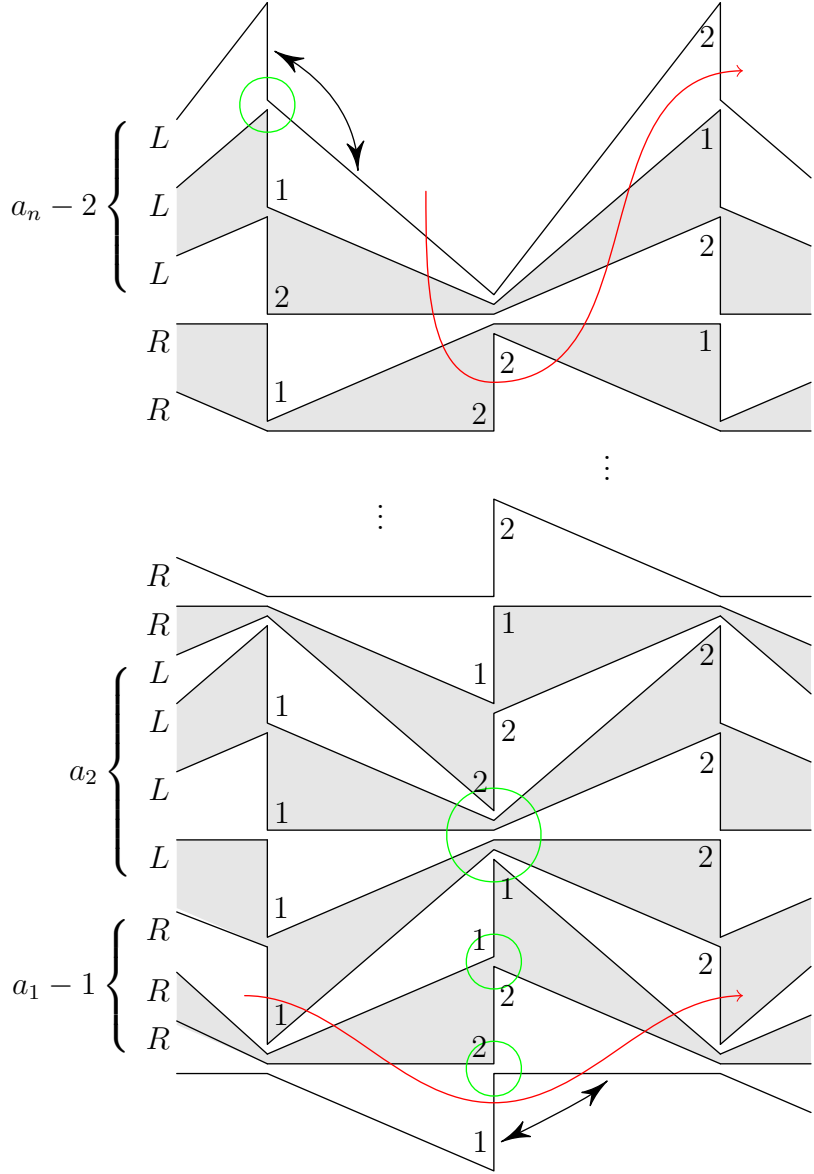}
  \caption{A    developing map  of one peripheral annulus
    for
    the case~\eqref{enum:case1}.
    Here $R$ and $L$ respectively denote tetrahedra $\triangle_i(R)$
    and $\triangle_i(L)$ in Fig.~\ref{fig:sphere_2triangle}.
    To  emphasize the layered structure, 
    each vertex is opened up   and
    each layer is colored alternately.
    See Fig.~\ref{fig:sphere_2triangle}.
    By $1$ and $2$ we denote dihedral angles $z_1[k]$ and $z_2[k]$
    respectively.
    Two
    left-right-arrows  denote an identification of edges,
    which are consequences
    of folding of $\sigma[2]$ and $\sigma[c-1]$.
  }
  \label{fig:cusp_sphere_1}
\end{figure}

We start a proof for the  case~\eqref{enum:case1}.
We recall that the $y$-pattern is
\begin{equation*}
  \boldsymbol{y}[1]\xrightarrow{R}
  \cdots \xrightarrow{R}
  \boldsymbol{y}[a_1]
  \xrightarrow{L}
  \boldsymbol{y}[a_1+1]\xrightarrow{L}\cdots \xrightarrow{L}
  \boldsymbol{y}[a_1+a_2]\xrightarrow{R} \cdots.
\end{equation*}
A typical developing map of one peripheral annulus
is depicted in Fig.~\ref{fig:cusp_sphere_1}.
First of all,
a gluing condition for a 
circle in the bottom of
Fig.~\ref{fig:cusp_sphere_1}
can be checked as follows;
\begin{equation*}
  \begin{aligned}
    z_2[2]  \,    z_1^{\prime \prime}[1] \, z_2^{\prime \prime}[1] 
    & =
    \frac{-1}{y[2]_2}\,
    \frac{1}{1+y[1]_1^{~-1}} \,    \frac{1}{1+y[1]_2^{~-1}}
    \\
    & =
    \frac{-1}{y[1]_6} = 1 ,
  \end{aligned}
\end{equation*}
where we have used
$\boldsymbol{y}[2]=R(\boldsymbol{y}[1])$ as given
in~\eqref{explicit_y_mutation}, and the last equality is
from~\eqref{folding_y_bottom}.
In the same manner, a gluing condition for the second
circle
from the bottom
can be checked by use of
$\boldsymbol{y}[3]=R(\boldsymbol{y}[2])$ and
$\boldsymbol{y}[2]=R(\boldsymbol{y}[1])$
as
\begin{equation*}
  \begin{aligned}
    z_1[3] \, z_1^{\prime\prime}[2] \, z_2^{\prime \prime}[2] \,
    z_2[1]
    &=
    \frac{1}{y[3]_1} \,
    \frac{1}{1+y[2]_1^{~-1}} \,
    \frac{1}{1+y[2]_2^{~-1}} \,
    \frac{1}{y[1]_2}
    \\
    &
    = \frac{1}{y[2]_5 \, y[1]_2}
    = 1 .
  \end{aligned}
\end{equation*}
We can check that
$z_1[k+1] \, z_1^{\prime\prime}[k] \, z_2^{\prime \prime}[k] \,
z_2[k-1]=1$,
and that
this type of identity denotes a gluing condition composed from four
angles in
Fig.~\ref{fig:cusp_sphere_1}.
A gluing condition for a 
circle in the middle of the figure
is also checked as follows,
\begin{align*}
  &
  z_2[a_1+a_2] \cdot
  \prod_{k=a_1-1}^{a_1+a_2-1} z_1^{\prime \prime}[k] \,
  z_2^{\prime\prime}[k]
  \cdot
  z_1[a_1-2]
  \\
  & = \frac{1}{y[a_1+a_2]_2} \cdot
  \prod_{k=a_1}^{a_1+a_2-1}
  \frac{1}{1+ y[k]_3^{~-1}} \, \frac{1}{1+y[k]_4^{~-1}}
  \cdot
\frac{1}{1+ y[a_1-1]_1^{~-1}} \, \frac{1}{1+y[a_1-1]_2^{~-1}} \,
  \cdot
  \frac{1}{y[a_1-2]_1}
  \\
  & = \frac{1}{y[a_1+a_2]_2} \cdot
  \prod_{k=a_1}^{a_1+a_2-1}
  \frac{1}{1+ y[k]_3^{~-1}} \, \frac{1}{1+y[k]_4^{~-1}}
  \cdot
  y[a_1]_2
  \\
  &
  = \frac{1}{y[a_1+a_2]_2} y[a_1+a_2]_2
  =1 .
\end{align*}
Other gluing conditions can be checked by the similar computations.

A completeness condition can be read from
curves in the figure.
From the lower curve we have
\begin{equation*}
  \begin{aligned}
    z_1[a_1]
    \cdot
    \prod_{k=2}^{a_1-1} z_1^\prime[k] \, z_2^\prime[k]
    \cdot
    \frac{z_2^\prime[1]}{z_1^{\prime\prime}[1]}
    & =
    \frac{-1}{y[a_1]_3} 
    \cdot
    \prod_{k=2}^{a_1-1}
    \left( 1+y[k]_1 \right) \, 
    \left(1+y[k]_2 \right)
    \cdot
    \frac{1}{y[1]_1}
    \\
    & =
    \frac{-1}{y[1]_1 \, y[1]_3}
    = 1 ,
  \end{aligned}
\end{equation*}
where we have used~\eqref{folding_y_bottom} in the last equality.
It is noted  that
the completeness condition can also be checked from
the  upper
curve in the figure as
\begin{equation*}
  \begin{aligned}
    &
    \frac{z_1^{\prime \prime}[c-3]}{z_2^\prime[c-3]}
    \cdot
    \prod_{k=c-a_n-1}^{c-4}
    z_1^{\prime \prime}[k] \, z_2^{\prime \prime}[k]
    \cdot
    z_2[c-a_n-2]
    \\
    & =
    -y[c-2]_2 \, y[c-2]_5
    = 1 ,
  \end{aligned}
\end{equation*}
where~\eqref{folding_y_top} and~\eqref{yyy}
are used in the last equality.

This completes a proof for the case~\eqref{enum:case1}.

\begin{figure}[htbp]
  \centering
  \includegraphics[]{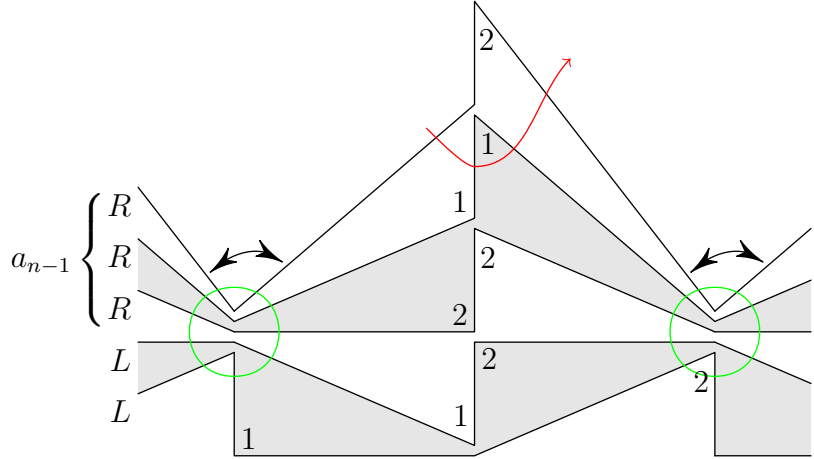}
  \caption{An upper part of developing map for 
    the case~\eqref{enum:case2}.}
  \label{fig:cusp_sphere_4}
\end{figure}

In the  case~\eqref{enum:case2},
the developing map is slightly  different from the case~\eqref{enum:case1}:
the lower part of Fig.~\ref{fig:cusp_sphere_1} is same, but an
upper part  is replaced with
Fig.~\ref{fig:cusp_sphere_4}.
A consistency check of gluing equations is similar to the
case~\eqref{enum:case1}.
For example,
a gluing equation for a
circle in Fig.~\ref{fig:cusp_sphere_4}
is 
\begin{align*}
  &
  \prod_{k=c-a_{n-1}-1}^{c-3}
  z_1^\prime[k] \, z_2^\prime[k]
  \cdot
  z_2[c-a_{n-1}-2]
  \\
  &   =
  \prod_{k=c-a_{n-1}}^{c-3}
  \left( 1 + y[k]_1 \right) \, \left( 1 + y[k]_2 \right)
  \cdot
  \left( 1 + y[c-a_{n-1}-1]_3 \right) \,
  \left( 1 + y[c-a_{n-1}-1]_4 \right) \,
  \frac{-1}{y[c-a_{n-1}-2]_4}
  \\
  & = - y[c-2]_3
  = 1  ,
\end{align*}
where the folding condition~\eqref{folding_y_top} is used in the last equality.
A
curve in
Fig.~\ref{fig:cusp_sphere_4}
is read as
\begin{equation*}
  \begin{aligned}
    \frac{z_1^{\prime \prime}[c-3]}{z_2^\prime[c-3]} \,
    z_1[c-4]
    &=
    -y[c-2]_5 \, y[c-2]_2
    =1 ,
  \end{aligned}
\end{equation*}
where the last equality is a consequence of~\eqref{folding_y_top}
and~\eqref{yyy}, and
this  denotes a completeness condition.
This completes a proof for the case~\eqref{enum:case2}.

\begin{figure}[htbp]
  \centering
  \includegraphics[]{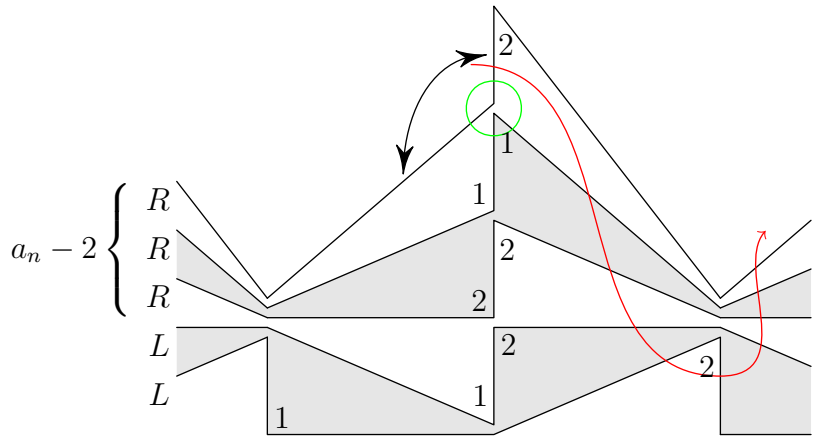}
  \caption{An upper part of developing map for 
    the case~\eqref{enum:case3}.}
  \label{fig:cusp_sphere_3}
\end{figure}

In the case~\eqref{enum:case3},
an upper part of Fig.~\ref{fig:cusp_sphere_1} is replaced with
Fig.~\ref{fig:cusp_sphere_3}.
A gluing condition for a
circle in Fig.~\ref{fig:cusp_sphere_3}
can be checked as
\begin{align*}
  z_1^{\prime \prime}[c-3] \,  z_2^{\prime \prime}[c-3]
  z_1[c-4]
  &=
  \left( 1+y[c-3]_1^{~-1} \right)^{-1} \,
  \left( 1+y[c-3]_2^{~-1} \right)^{-1} \,
  \frac{-1}{y[c-4]_1}
  \\
  & = 
  - y[c-2]_2
  = 1 .
\end{align*}
Correspondingly
a completeness condition is read from a
curve in
Fig.~\ref{fig:cusp_sphere_3}, and we can check
\begin{equation*}
  \begin{aligned}
    \frac{z_1^\prime[c-3]}{z_2^{\prime\prime}[c-3]}
    \cdot
    \prod_{k=c-a_n-1}^{c-4}
    z_1^\prime[k] \, z_2^\prime[k]
    \cdot
    z_1[c-a_n-2]
    & 
    = - y[c-2]_6 \, y[c-2]_4
    = 1,
  \end{aligned}
\end{equation*}
by using~\eqref{folding_y_top} and~\eqref{yyy}
at the last equality.
This completes a proof for the case~\eqref{enum:case3}.

\begin{figure}[htbp]
  \centering
  \includegraphics[]{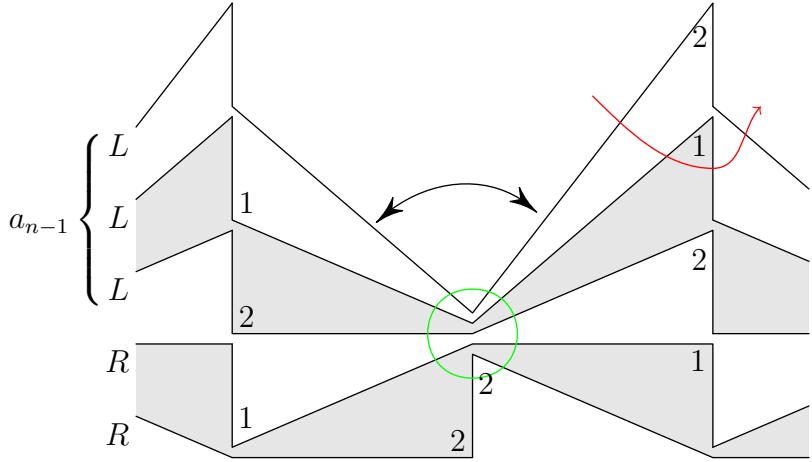}
  \caption{An upper part of developing map for 
    case~\eqref{enum:case4}.}
  \label{fig:cusp_sphere_2}
\end{figure}

In the last   case~\eqref{enum:case4}, an upper part of
Fig.~\ref{fig:cusp_sphere_1} is replaced by
Fig.~\ref{fig:cusp_sphere_2}.
A gluing condition for a
curve in Fig.~\ref{fig:cusp_sphere_2}
is checked as
\begin{align*}
  &
  \prod_{k=c-a_{n-1}-1}^{c-3}
  z_1^{\prime \prime}[k] \, z_2^{\prime \prime}[k]
  \cdot
  z_2[c-a_{n-1}-2]
  \\
  & =
  \prod_{k=c-a_{n-1}}^{c-3}
  \frac{1}{
    1+ y[k]_3^{~-1}} \,
  \frac{1}{1+y[k]_4^{~-1}}
  \cdot
  \frac{1}{
    1+ y[c-a_{n-1}-1]_1^{~-1}} \,
  \frac{1}{1+y[c-a_{n-1}-1]_2^{~-1}} \,
  \frac{-1}{y[c-a_{n-1}-2]_2}
  \\
  & =
  \prod_{k=c-a_{n-1}}^{c-3}
  \left(   1+ y[k]_3^{~-1}\right)^{-1} \,
  \left(1+y[k]_4^{~-1} \right)^{-1}
  \cdot
  \left( - y[c-a_{n-1}]_1 \right)
  \\
  & =
  -y[c-2]_1
  = 1 .
\end{align*}
Correspondingly 
we have  a completeness condition corresponding to  a
curve,
\begin{align*}
  \frac{z_1^\prime[c-3]}{z_2^{\prime\prime}[c-3]} \,
  z_2[c-4] 
  & =
  \frac{1+y[c-3]_3}{y[c-3]_3} \,
  \left( 1 + y[c-3]_4 \right) \, \frac{-1}{y[c-4]_4}
  \\
  &
  =
  - y[c-2]_3 \, y[c-2]_6
  = 1,
\end{align*}
where we have used~\eqref{folding_y_top} and~\eqref{yyy}
at the last equality.
This completes a proof for the case~\eqref{enum:case4}.

To conclude, we have checked that both the gluing conditions and the
completeness condition are fulfilled for every cases~\eqref{enum:case1}--\eqref{enum:case4}.

\subsection{Cluster Pattern and Complex Volume}
To study the complex volume of the complement of the $2$-bridge link
$K_{q/p}$,
we  reformulate the preceding result by use of the cluster variable~$\boldsymbol{x}$.
Recall that we use the tropical semifield in Def.~\ref{def:P} for the
coefficient~$\boldsymbol{\varepsilon}$.
Actions of the  flips $R$ and $L$ defined as permuted
cluster mutations~\eqref{sphere_mutation}
are explicitly given respectively by
\begin{align}
  &
  (\boldsymbol{x}, \boldsymbol{\varepsilon})
  \xrightarrow{R}
  R(\boldsymbol{x},
  \boldsymbol{\varepsilon})
  ,
  &
  &
  (\boldsymbol{x}, \boldsymbol{\varepsilon})
  \xrightarrow{L}
  L(\boldsymbol{x},
  \boldsymbol{\varepsilon}) ,
\end{align}
where
\begin{equation}
\label{explicit_sphere_xe}
\begin{aligned}
  &
  R(\boldsymbol{x}, \boldsymbol{\varepsilon})
  =
  \left(
  \begin{pmatrix}
    x_5 \\
    x_6 \\
    x_3 \\
    x_4 \\[1ex]
    \displaystyle
    \frac{\ve_2}{1 \oplus \ve_2} \, \frac{x_3\, x_4}{x_2} 
    +\frac{1}{1 \oplus \ve_2} \frac{x_5 \, x_6}{x_2}
    \\[2ex]
    \displaystyle
    \frac{\ve_1}{1 \oplus \ve_1} \, \frac{x_3\, x_4}{x_1} 
    +\frac{1}{1 \oplus \ve_1} \frac{x_5 \, x_6}{x_1}
  \end{pmatrix}^\top ,
  \begin{pmatrix}
    \varepsilon_5 
    \displaystyle
    \frac{\ve_1}{1 \oplus \ve_1} \frac{\ve_2}{1 \oplus \ve_2}\\[2ex]
    \varepsilon_6 
    \displaystyle
    \frac{\ve_1}{1 \oplus \ve_1} \frac{\ve_2}{1 \oplus \ve_2}\\[2ex]
    \varepsilon_3 (1 \oplus \ve_1)(1 \oplus \ve_2)\\[1ex]
    \varepsilon_4 (1 \oplus \ve_1)(1 \oplus \ve_2)\\
    \varepsilon_2^{-1} \\
    \varepsilon_1^{-1} 
  \end{pmatrix}^\top 
  \right),
  \\
  &
  L(\boldsymbol{x}, \boldsymbol{\varepsilon})
  =
  \left(
  \begin{pmatrix}
    x_1 \\
    x_2 \\
    x_5 \\
    x_6 \\[1ex]
    \displaystyle
    \frac{1}{1 \oplus \ve_4} \, \frac{x_1 \, x_2}{x_4}
    +\frac{\varepsilon_4}{1 \oplus \ve_4 } \,    \frac{x_5 \, x_6}{x_4}    
    \\[2ex]
    \displaystyle
    \frac{1}{1 \oplus \ve_3} \, \frac{x_1 \, x_2}{x_3}
    +\frac{\varepsilon_3}{1 \oplus \ve_3 } \,    \frac{x_5 \, x_6}{x_3}    
  \end{pmatrix}^\top ,
  \begin{pmatrix}
    \varepsilon_1
    \displaystyle 
    \frac{\ve_3}{1 \oplus \ve_3} \frac{\ve_4}{1 \oplus \ve_4}\\[2ex]
    \varepsilon_2
    \displaystyle 
    \frac{\ve_3}{1 \oplus \ve_3} \frac{\ve_4}{1 \oplus \ve_4}\\[2ex]
    \varepsilon_5 (1 \oplus \ve_3) (1 \oplus \ve_4)\\[1ex]
    \varepsilon_6 (1 \oplus \ve_3) (1 \oplus \ve_4)\\
    \varepsilon_4^{-1} \\
    \varepsilon_3^{-1} 
  \end{pmatrix}^\top 
\right) .
\end{aligned}
\end{equation}

\begin{definition}
  A cluster pattern of $K_{q/p}$ is
  $(\boldsymbol{x}[k], \boldsymbol{\varepsilon}[k])$
  for $k=1,2, \dots, c-2$ defined recursively by
  \begin{equation}
        (  \boldsymbol{x}[k+1],
    \boldsymbol{\varepsilon}[k+1] )
    =
    F_k\left(
      \boldsymbol{x}[k] ,
      \boldsymbol{\varepsilon}[k]
    \right) ,
  \end{equation}
  where $F_k$ is $R$ or $L$ as~\eqref{flip_sphere}.
\end{definition}

Prop.~\ref{thm:y_from_xe} shows that the $y$-pattern
$\boldsymbol{y}[k]$
in 
Prop.~\ref{thm:volume_sphere} can be given in terms of the cluster
pattern 
$(\boldsymbol{x}[k], \boldsymbol{\varepsilon}[k])$
as~\eqref{y_from_xe},
\begin{multline}
  \label{eq:y-vex}
  \boldsymbol{y}[k]
  =
  \left(
    \varepsilon[k]_1 \, \frac{x[k]_3 \, x[k]_4}{x[k]_5 \, x[k]_6},
    \varepsilon[k]_2 \, \frac{x[k]_3 \, x[k]_4}{x[k]_5 \, x[k]_6},
    \varepsilon[k]_3 \, \frac{x[k]_5 \, x[k]_6}{x[k]_1 \, x[k]_2},
    \right.
    \\
    \left.
    \varepsilon[k]_4 \, \frac{x[k]_5 \, x[k]_6}{x[k]_1 \, x[k]_2},
    \varepsilon[k]_5 \, \frac{x[k]_1 \, x[k]_2}{x[k]_3 \, x[k]_4},
    \varepsilon[k]_6 \, \frac{x[k]_1 \, x[k]_2}{x[k]_3 \, x[k]_4}
  \right) .
\end{multline}

The  conditions~\eqref{folding_y_bottom} and~\eqref{folding_y_top} for
the $y$-variables,
which come from
the  folding of the Farey triangles $\sigma[2]$ and
$\sigma[c-1]$,
are fulfilled when we suppose the following
constraints on the cluster variables~$\boldsymbol{x}[k]$
and the coefficients~$\boldsymbol{\varepsilon}[k]$:
\begin{gather}
  \label{eq:init-ve}
  \psi(\boldsymbol{\varepsilon}[1])
  =
  ( -1, -1, 1, 1, -1, -1),
  \quad
  \text{or}
  \quad
  (1,1, -1, -1, -1, -1) ,
  \\[1ex]
  \label{sphere_x_initial}
  \boldsymbol{x}[1]
  =
  (x,x,x,x,x_5,x_6) ,
\end{gather}
and
\begin{gather}
  \label{folding_ve_top}
  \begin{cases}
    \psi(\varepsilon[c-2]_3)=\psi(\varepsilon[c-2]_4)=-1 ,
    & \text{if $n$ is even,}
    \\[1ex]
    \psi(\varepsilon[c-2])_2=\psi(\varepsilon[c-2]_1)=-1 ,
    & \text{if $n$ is odd,}
  \end{cases}
  \\[2ex]
  \label{folding_x_top}
  \begin{cases}
    x[c-2]_1=x[c-2]_2=x[c-2]_5=x[c-2]_6 ,
    & \text{if $n$ is even},
    \\[1ex]
    x[c-2]_3=x[c-2]_4=x[c-2]_5=x[c-2]_6 ,
    & \text{if $n$ is odd}.
  \end{cases}
\end{gather}
The initial cluster $x$-variable~\eqref{sphere_x_initial} 
means that
the edges for $x_1$, $x_2$, $x_3$, and $x_4$ are identified.
Here $\psi$ is defined in Def.~\ref{def:beta}.
At \eqref{folding_x_top} and in the rest of this section,
the coefficient~${\varepsilon}[k]_i$ in the cluster
variable~${x}[k]_i$
are  set to be~$+1$ or~$-1$
by acting the map~$\psi$. 
We should note that
one of the two conditions~\eqref{eq:init-ve} 
is chosen so that it is consistent with the 
constraint~\eqref{folding_ve_top}.
To fulfill~\eqref{eq:init-ve}, we set
\begin{equation*}
  \boldsymbol{\ve}[1] 
  = 
  (\delta^{m_1},\delta^{m_1},\delta^{m_2},\delta^{m_2},\delta^{m_3},\delta^{m_3}) ,
\end{equation*}
where
$(m_1, m_2, m_3)$ is
$(\text{odd, even, odd})$
(resp. $(\text{even, odd, odd})$)
for the first (resp. second) case.
%
By these conditions, the cluster variables 
$x[k]_j$
are solved up to constant.
We can see from the three-dimensional interpretation of
the  flips in
Fig.~\ref{fig:sphere_2triangle} that the cluster
variables~$x[k]_j$ are assigned to edges
of the ideal tetrahedra $\triangle_i(F_k)$,
and that
both~\eqref{sphere_x_initial} and~\eqref{folding_x_top} imply
that
the identified edges in~$S^3\setminus K_{q/p}$
have  equal  complex numbers.
Then the cluster variables are identified with Zickert's complex
variables on edges, and
the complex volume can be computed from the cluster variables.

To compute the complex volume modulo $\pi^2$, we should  fix an
orientation
of each ideal tetrahedra~$\triangle_i(F_k)$.
Although, unlike the case of
the once-punctured torus bundle, 
an orientation of triangulations of $\Sigma_{0,4}$ cannot be fixed
uniquely due to the transformation group $\Gamma$, and
it is tedious
to give orientations  from $q/p$.
So for simplicity, we discard a vertex ordering, and we
give a formula of   the complex
volume modulo~$\frac{\pi^2}{6}$.
From \eqref{sphere_z_from_y} and \eqref{eq:y-vex}, we obtain
the following.

\begin{lemma}\label{lem:z-vex}
  Let $(\boldsymbol{x}[k], \boldsymbol{\varepsilon}[k])$ 
  be a cluster pattern which satisfies the conditions
  \eqref{eq:init-ve}--\eqref{folding_x_top}.
  Then the moduli $z_1[k]$ and $z_2[k]$ of a pair of tetrahedra
  $\triangle_1(F_k)$ and $\triangle_2(F_k)$
  are  given from the cluster pattern,
  \begin{equation}\label{eq:z-vex}
    z_i[k]
    =
    \begin{cases}
      \displaystyle
      - \frac{1}{\varepsilon[k]_i} \,
      \frac{x[k]_5 \, x[k]_6}{x[k]_3 \, x[k]_4} ,
      &
      \text{for $F_k=R$,}
      \\[3ex]
      \displaystyle
      - \frac{1}{\varepsilon[k]_{2+i}} \,
      \frac{x[k]_1 \, x[k]_2}{x[k]_5 \, x[k]_6} ,
      &
      \text{for $F_k=L$.}
    \end{cases}
  \end{equation}
  for $k=1,\cdots,c-2$ and $i=1,2$.
\end{lemma}

We choose $x$-variable such that the $y$-variables are geometric.
See Remark~\ref{rem:cv_bridge}.
Then we can identify the $x$-variables with the edge parameters
$c_{ab}$,
and  we obtain
the flattenings of the tetrahedra $\triangle_i(F_k)$ as follows.

\begin{lemma}\label{lem:2bridge-flat}
  We follow the setting of Lemma \ref{lem:z-vex}.
  The flattening
  $(z_i[k]; p_i[k], q_i[k])$ for the
  tetrahedron $\triangle_i(F_k)$ are
  given by
  \begin{equation}
    \begin{aligned}
      &\log z_i[k] + p_i[k] \, \pi \, \I
      =
      \begin{cases}
        \log x[k]_5 + \log x[k]_6 - \log x[k]_3 - \log x[k]_4,
        & \text{for $F_k=R$,}
        \\[2ex]
        \log x[k]_1 + \log x[k]_2 - \log x[k]_5 - \log x[k]_6,
        & \text{for $F_k=L$,}
      \end{cases}
      \\[3ex]
      &-\log\left(1-z_i[k]\right) +q_i[k] \, \pi \,\I
      \\
      &\qquad \qquad
      =
      \begin{cases}
        \log x[k]_3 + \log x[k]_4 - \log x[k]_i - \log x[k+1]_{7-i} ,
        & \text{for $F_k=R$.}
        \\[2ex]
        \log x[k]_5 + \log x[k]_6 - \log x[k]_{2+i} - \log
        x[k+1]_{7-i},
        & \text{for $F_k=L$,}
      \end{cases}
    \end{aligned}
  \end{equation}
for $k=1\cdots,c-2$ and $i=1,2$.
\end{lemma}

\begin{proof}
  Due to~\eqref{explicit_sphere_xe}, we have
  \begin{equation}\label{eq:1-z-vex}
    \frac{1}{1-z_i[k]}
    =
    \begin{cases}
      \displaystyle
      \frac{\varepsilon[k]_i}{1 \oplus \varepsilon[k]_i}
      \,
      \frac{x[k]_3 \, x[k]_4}{x[k]_i \, x[k+1]_{7-i}},
      &
      \text{for $F_k=R$,}
      \\[3ex]
      \displaystyle
      \frac{\varepsilon[k]_{2+i}}{1 \oplus \varepsilon[k]_{2+i}}
      \,
      \frac{x[k]_5 \, x[k]_6}{
        x[k]_{2+i} \, x[k+1]_{7-i}}
      &
      \text{for $F_k=L$,}
    \end{cases}
  \end{equation}
  for $i=1,2$.
  By comparing~\eqref{eq:z-vex} and~\eqref{eq:1-z-vex}
  with~\eqref{Zickert_c}, we get the claim.
\end{proof}



It is straightforward to get the following theorem from these Lemmas.

\begin{theorem}
  \label{thm:complex_volume_2bridge}
  The complex volume of the $2$-bridge link complement
  $S^3 \setminus K_{q/p}$ is given by
  the flattenings of Lemma \ref{lem:2bridge-flat} as 
  \begin{equation}
    \I \, \left(
      \Vol(S^3 \setminus K_{q/p})
      + \I \,
      \CS(S^3 \setminus K_{q/p})
    \right)
    =
    \sum_{k=1}^{c-3}
    \sum_{i=1,2}
    \widehat{L}(z_i[k]; p_i[k], q_i[k])
    \mod \frac{\pi^2}{6}.
  \end{equation}
\end{theorem}


\begin{remark}
  \label{rem:Ronly}
  The orientation of each tetrahedron can be determined case by case.
  As an example,
  we study
  $K_{[a+1,2]}$, which  has a  cluster pattern
  \begin{equation*}
    \left(\boldsymbol{x}[1], \boldsymbol{\varepsilon}[1]\right)
    \xrightarrow{R}
    \left(\boldsymbol{x}[2], \boldsymbol{\varepsilon}[2]\right)
    \xrightarrow{R}
    \cdots
    \xrightarrow{R}
    \left(\boldsymbol{x}[a+1], \boldsymbol{\varepsilon}[a+1]\right) .
  \end{equation*}
  Based on triangulations of link complements, we may choose vertex
  ordering so that
  \begin{equation*}
    \left( \sign(\triangle_1(F_k)) , \sign(\triangle_2(F_k)) \right)
    =
    \begin{cases}
      (-1,-1),
      &
      \text{for odd $k$ and $k\neq a, a-1$,}
      \\
      (+1,+1),
      &
      \text{for even $k$ and $k\neq a$,}
      \\
      (+1,-1),
      &
      \text{for odd $k$ and $k = a, a-1$,}
      \\
      (-1,-1) ,
      &
      \text{for even $k$ and $k= a$.}
    \end{cases}
  \end{equation*}
  We can then identify
  \begin{multline}
    (z_1[k] , z_2[k])
    \\
    =
    \begin{cases}
      \displaystyle
      \left(
        \frac{1 \oplus \varepsilon[k]_1}{\varepsilon[k]_1}
        \, 
        \frac{x[k]_1 \, x[k+1]_6}{
          x[k]_3 \, x[k]_4},
        \frac{1 \oplus \varepsilon[k]_2}{\varepsilon[k]_2}
        \, 
        \frac{x[k]_2 \, x[k+1]_5}{
          x[k]_3 \, x[k]_4}      
      \right) ,
      &
      \text{for odd $k$ and $k\neq a, a-1$,}
      \\[2ex]
      \displaystyle
      \left(
        - \frac{1}{\varepsilon[k]_1} \, 
        \frac{x[k]_5 \, x[k]_6}{
          x[k]_3 \, x[k]_4},
        -\frac{1}{\varepsilon[k]_2} \, 
        \frac{x[k]_5 \, x[k]_6}{
          x[k]_3 \, x[k]_4}
      \right),
      &
      \text{for even $k$ and $k\neq a$,}
      \\[2ex]
      \displaystyle
      \left(
        \left( 1 \oplus \varepsilon[k]_1  \right) \,
        \frac{x[k]_1 \, x[k+1]_6}{
          x[k]_5 \, x[k]_6},
        \frac{1 \oplus \varepsilon[k]_2}{\varepsilon[k]_2}
        \,
        \frac{x[k]_2 \, x[k+1]_5}{
          x[k]_3 \, x[k]_4}
      \right) ,
      &
      \text{for odd $k$ and $k = a, a-1$,}
      \\[2ex]
      \displaystyle
      \left(
        \frac{1}{1 \oplus \varepsilon[k]_1} \,
        \frac{x[k]_5 \, x[k]_6}{
          x[k]_1 \, x[k+1]_6},
        \frac{1}{1 \oplus \varepsilon[k]_2} \,
        \frac{x[k]_5 \, x[k]_6}{
          x[k]_2 \, x[k+1]_5}      
      \right),
      &
      \text{for even $k$ and $k= a$.}
    \end{cases}
    \label{twist_p_data}
  \end{multline}
  Accordingly, we have
  \begin{multline}
    \left(
      \frac{1}{1-z_1[k]} ,
      \frac{1}{1- z_2[k]}
    \right)
    \\
    =
    \begin{cases}
      \displaystyle
      \left(
        - \varepsilon[k]_1 \, 
        \frac{x[k]_3 \, x[k]_4}{
          x[k]_5 \, x[k]_6},
        -\varepsilon[k]_2 \, 
        \frac{x[k]_3 \, x[k]_4}{
          x[k]_5 \, x[k]_6}      
      \right) ,
      &
      \text{for odd $k$ and $k\neq a, a-1$,}
      \\[2ex]
      \displaystyle
      \left(
        \frac{\varepsilon[k]_1}{1 \oplus \varepsilon[k]_1}
        \, 
        \frac{x[k]_3 \, x[k]_4}{
          x[k]_1 \, x[k+1]_6},
        \frac{\varepsilon[k]_2}{1 \oplus \varepsilon[k]_2}
        \, 
        \frac{x[k]_3 \, x[k]_4}{
          x[k]_2 \, x[k+1]_5}
      \right),
      &
      \text{for even $k$ and $k\neq a$,}
      \\[2ex]
      \displaystyle
      \left(
        -\frac{1}{\varepsilon[k]_1} \,
        \frac{x[k]_5 \, x[k]_6}{
          x[k]_3 \, x[k]_4},
        -\varepsilon[k]_2 \,
        \frac{x[k]_3 \, x[k]_4}{
          x[k]_5 \, x[k]_6}
      \right) ,
      &
      \text{for odd $k$ and $k = a, a-1$,}
      \\[2ex]
      \displaystyle
      \left(
        \frac{1 \oplus \varepsilon[k]_1}{\varepsilon[k]_1}
        \,
        \frac{x[k]_1 \, x[k+1]_6}{
          x[k]_3 \, x[k]_4},
        \frac{1 \oplus \varepsilon[k]_2}{\varepsilon[k]_2}
        \,
        \frac{x[k]_2 \, x[k+1]_5}{
          x[k]_3 \, x[k]_4}      
      \right),
      &
      \text{for even $k$ and $k= a$.}
    \end{cases}
    \label{twist_q_data}
  \end{multline}
  These data give the flattening
  $(z_i[k]; p_i[k], q_i[k])$
  for each tetrahedron from~\eqref{Zickert_c}, and we get
  \begin{align}
    \label{twist_cv}
    \begin{split}
    &\I \, \left(
      \Vol(S^3 \setminus K_{[a+1,2]})
      + \I \, \CS(S^3 \setminus K_{[a+1,2]})
    \right)
    \\
    & \hspace{3cm}
    =
    \sum_{k=1}^a
    \sum_{i=1,2}
    \sign(\triangle_i(F_k)) \,
    \widehat{L}(z_i[k]; p_i[k], q_i[k])
    \mod \pi^2 .
    \end{split}
  \end{align}
\end{remark}

\subsection{\mathversion{bold}Example: $6_1$}
The knot $6_1$ is a two-bridge knot $K_{[4,2]}$, which corresponds to
$a=3$ in Remark~\ref{rem:Ronly}.
When we set
an initial seed as
\begin{align*}
  &\boldsymbol{x}[1]=(1,1,1,1,x,x),
  &
  \boldsymbol{\varepsilon}[1]
  =
  (\delta^0, \delta^0, \delta^1,\delta^1,\delta^1,\delta^1),
\end{align*}
the constraint~\eqref{folding_x_top} is read as
\begin{equation*}
  x \,(2+x^2) =
  1+3 \, x^2+ x^4 .
\end{equation*}
A geometric solution is $x= 0.1048 \cdots- \I \cdot 1.5524\cdots$, and the
flattenings are calculated from~\eqref{twist_p_data}
and~\eqref{twist_q_data} to be
\begin{equation*}
  \begin{array}{r|ccc}
    k & (z_1[k],z_2[k]) &  (p_1[k],p_2[k]) & (q_1[k],q_2[k])
    \\
    \hline
    1 &
    (-1.3992 \cdots - \I \cdot 0.3256 \cdots,
    -1.3992 \cdots - \I \cdot 0.3256 \cdots )
    &(0,0) &(1,1)
    \\
    2 &
    (1.8518\cdots + \I \cdot 0.9112 \cdots,
    1.8518\cdots + \I \cdot 0.9112 \cdots)
    &(-2,-2) &(-1,-1)
    \\
    3 &
    (0.8951\cdots + \I \cdot 1.552 \cdots,
    0.9566\cdots - \I \cdot 0.6412\cdots) &
    (-2,0) & (1,-1)
  \end{array}
\end{equation*}
We get from~\eqref{twist_cv}
\begin{equation*}
  \Vol({S^3 \setminus 6_1}) + \I \, \CS(S^3 \setminus 6_1)
  =
  3.1639 \cdots + \I \cdot 3.0788 \cdots
  \in
  \mathbb{C}/ \I \, \pi^2 \, \mathbb{Z},
\end{equation*}
which coincides with the known result~\cite{SnapPy}.

\section*{Acknowledgments}
One of the authors (KH) thanks T.~Dimofte and J.~Murakami for
communications.
He also thanks to
participants of the workshop
``Low-Dimensional Topology and Number Theory'' at  MFO in 2012.
Thanks are also to M.~Sakuma for comments on ideal triangulations.
The work of KH is supported in part by JSPS KAKENHI Grant Number~23340115,
24654041, 22540069.
RI is partially supported by 
Grant-in-Aid for Young Scientists (B) (22740111).

\end{document}